\documentclass[11pt]{amsart}

\usepackage{epigamath-voisin}

\usepackage[english]{babel}

\numberwithin{equation}{section}

\usepackage{enumitem}
\usepackage{tikz-cd}
\tikzset{commutative diagrams/arrow style=Latin Modern}
\usepackage[capitalise]{cleveref}

\newenvironment{aenumerate}{%
	\begin{enumerate}[label=(\alph{*}), ref=(\alph{*})]
}{%
	\end{enumerate}%
}

\newtheorem{theorem}{Theorem}[section]
\newtheorem{lemma}[theorem]{Lemma}
\newtheorem{corollary}[theorem]{Corollary}
\newtheorem{proposition}[theorem]{Proposition}

\theoremstyle{definition}
\newtheorem{definition}[theorem]{Definition}

\theoremstyle{remark}
\newtheorem{example}[equation]{Example}
\newtheorem*{note}{Note}

\newcommand{\newpar}{\subsection{}}
\crefname{subsection}{\textsection}{\textsection}

\newcommand{\shH}{\mathcal{H}}

\newcommand{\norm}[1]{\lVert#1\rVert}

\newcommand{\abs}[1]{\lvert #1 \rvert}

\newcommand{\tensor}{\otimes}

\newcommand{\NN}{\mathbb{N}}
\newcommand{\ZZ}{\mathbb{Z}}
\newcommand{\QQ}{\mathbb{Q}}
\newcommand{\RR}{\mathbb{R}}
\newcommand{\CC}{\mathbb{C}}
\newcommand{\HH}{\mathbb{H}}
\newcommand{\PP}{\mathbb{P}}

\newcommand{\menge}[2]{\bigl\{ \thinspace #1 \thinspace\thinspace \big\vert%
\thinspace\thinspace #2 \thinspace \bigr\}}

\DeclareMathOperator{\rk}{rk}

\DeclareMathOperator{\Res}{Res}

\DeclareMathOperator{\id}{id}

\renewcommand{\Im}{\operatorname{Im}}
\renewcommand{\Re}{\operatorname{Re}}

\DeclareMathOperator{\End}{End}

\DeclareMathOperator{\GL}{GL}
\DeclareMathOperator{\SL}{SL}

\newcommand{\define}[1]{\emph{#1}}

\newcommand{\sltwo}{\mathfrak{sl}_2(\CC)}

\newcommand{\shf}[1]{\mathscr{#1}}
\newcommand{\OX}{\shf{O}_X}
\newcommand{\OmX}{\Omega_X}

\newcommand{\shHZ}{\shH_{\ZZ}}

\newcommand{\restr}[1]{\big\vert_{#1}}

\newcommand{\argbl}{-}

\def\overbar#1#2#3{{%
	\setbox0=\hbox{\displaystyle{#1}}%
	\dimen0=\wd0
	\advance\dimen0 by -#2 
	\vbox {\nointerlineskip \moveright #3 \vbox{\hrule height 0.3pt width \dimen0}%
		\nointerlineskip \vskip 1.5pt \box0}%
}}

\newcommand{\dst}{\Delta^{\ast}}

\newcommand{\into}{\hookrightarrow}

\newcommand{\inner}[2]{\langle #1, #2 \rangle}

\newcommand{\HR}{H_{\RR}}
\newcommand{\HZ}{H_{\ZZ}}
\newcommand{\HC}{H_{\CC}}

\newcommand{\fu}{f^{\ast}}

\newcommand{\piu}{\pi^{\ast}}

\newcommand{\shE}{\shf{E}}

\newcommand{\shO}{\shf{O}}

\newcommand{\shHZx}{\shH_{\ZZ,x}}
\newcommand{\RRanexp}{\RR_{\mathrm{an}, \mathrm{exp}}}
\newcommand{\RRan}{\RR_{\mathrm{an}}}

\newcommand{\RRalg}{\RR_{\mathrm{alg}}}
\newcommand{\HQ}{H_{\QQ}}
\DeclareMathOperator{\Aut}{Aut}
\newcommand{\Sieg}{\mathfrak{S}}
\newcommand{\Gtl}{\tilde{G}}
\newcommand{\Ktl}{\tilde{K}}
\newcommand{\Stl}{\tilde{S}}
\newcommand{\Atl}{\tilde{A}}
\newcommand{\Ptl}{\tilde{P}}
\newcommand{\Utl}{\tilde{U}}
\newcommand{\Mtl}{\tilde{M}}
\newcommand{\ktl}{\tilde{k}}
\newcommand{\ptl}{\tilde{p}}
\DeclareMathOperator{\diag}{diag}

\newcommand{\shEt}{\tilde{\shE}}
\newcommand{\Et}{\tilde{E}}
\newcommand{\dt}{\mathit{dt}}

\newcommand{\Xt}{\tilde{X}}
\newcommand{\Phit}{\tilde{\Phi}}
\newcommand{\Fsh}{F_{\sharp}}
\newcommand{\Csh}{C_{\sharp}}
\newcommand{\half}{\frac{1}{2}}

\newcommand{\OmY}{\Omega_Y}
\newcommand{\OY}{\shO_Y}

\YearArticle{2023} \EpigaArticleNr{1}
\ReceivedOn{May 31, 2022}
\InFinalFormOn{February 1, 2023}
\AcceptedOn{March 6, 2023}

\title{Finiteness for self-dual classes in integral\\variations of Hodge structure}
\titlemark{Finiteness for self-dual classes in integral variations of Hodge structure}

\author{Benjamin Bakker}
\address{Department of Mathematics, Statistics,
and Computer Science, University of Illinois at Chicago, 851 S. Morgan St., Chicago, IL 60607}
\email{bakker.uic@gmail.com}

\author{Thomas W.~Grimm}
\address{Institute for Theoretical Physics, Utrecht University, Princetonplein 5, 3584 CE Utrecht, The Netherlands}
\email{t.w.grimm@uu.nl}

\author{Christian Schnell}
\address{Department of Mathematics, Stony Brook University, Stony Brook, NY 11794}
\email{christian.schnell@stonybrook.edu}

\author{Jacob Tsimerman}
\address{Department of Mathematics, University of Toronto, 215 Huron Street, Toronto, Canada}
\email{jacobt@math.toronto.edu}

\authormark{B.~Bakker, T.~W.~Grimm, C.~Schnell, and J.~Tsimerman}

\AbstractInEnglish{We generalize the finiteness theorem for the locus of Hodge classes with fixed self-intersection number, due to Cattani, Deligne, and Kaplan, from Hodge classes to self-dual classes. The proof uses the definability of period mappings in the o-minimal structure $\RRanexp$.}



\acknowledgement{During the preparation of this paper, B.B. was partially supported by NSF grant DMS-1848049.
The research of Th.G. is in part supported by the Dutch Research Council (NWO) 
via a Start-Up grant and a Vici grant. Ch.S. is partially supported by NSF grant
DMS-1551677 and by a Simons Fellowship from the Simons Foundation. He thanks the
Max-Planck-Institute for providing him with excellent working conditions during his
stay in Bonn.}

\begin{document}



\maketitle

\begin{prelims}

\DisplayAbstractInEnglish

\bigskip

\DisplayKeyWords

\medskip

\DisplayMSCclass







\end{prelims}


\newpage

\setcounter{tocdepth}{1}

\tableofcontents


\section{Introduction}

\newpar 
The purpose of this paper is to prove a rather unexpected new finiteness
result for polarized integral variations of Hodge structure, containing the theorem
of Cattani, Deligne, and Kaplan for the locus of Hodge classes \cite{Cattani+Deligne+Kaplan} as a
special case. Instead of integral Hodge classes, we consider integral classes that
are ``self-dual'', meaning that they are preserved by the action of the Weil
operator; the motivation for this comes from considerations in theoretical physics.
Analyzing such classes using the methods in \cite{Cattani+Deligne+Kaplan} becomes rather complicated, so
our main tool is going to be the definability of period mappings in the o-minimal
structure $\RRanexp$, recently proved by Bakker, Klingler, and Tsimerman
\cite{Bakker+Klingler+Tsimerman:TameTopology}. This transforms the problem into a pleasant set of exercises about
certain algebraic groups.

\newpar
We begin by describing a toy case of the problem, to set up the notation. Suppose
that $H$ is a polarized integral Hodge structure of even weight $2k$.
We denote by
\[
	\HC = \HZ \tensor_{\ZZ} \CC = \bigoplus_{p+q=2k} H^{p,q}
\]
the Hodge decomposition, and by $Q \colon \HZ \tensor_{\ZZ} \HZ \to \ZZ$ the
symmetric bilinear form giving the polarization. If we define the \define{Weil
operator} by the formula
\[
	Cv = i^{p-q} v \quad \text{for $v \in H^{p,q}$},
\]
then $C \in \End(\HR)$ and $C^2 = \id$, and the expression
\[
	\inner{\argbl}{\argbl} \colon \HR \tensor_{\RR} \HR \to \RR, \quad
	\inner{v}{w} = Q(v, Cw),
\]
puts a positive definite inner product on $\HR = \HZ \tensor_{\ZZ} \RR$. We usually
write the resulting \define{Hodge norm} simply as $\norm{v}^2 = \inner{v}{v}$. 

We shall be interested in integral vectors $v \in \HZ$ with the property that $Cv = v$. 
Since $C^2 = \id$, any vector $v \in \HR$ can of course be decomposed uniquely as
\[
	v = v_+ + v_- \qquad \text{with $C v_+ = v_+$ and $C v_- = - v_-$;}
\]
concretely, $v_+$ is the sum of all the ``even'' components in the Hodge decomposition
of $v$, and $v_-$ the sum of all the ``odd'' ones. By analogy with the action of the
Hodge $\ast$-operator on the cohomology of four-manifolds, we may call $v_+$ and
$v_-$ the \emph{self-dual} respectively \emph{anti-self-dual} part of $v$. We say
that an integral class is self-dual if $v = v_+$; note that Hodge classes (in $\HZ
\cap H^{k,k}$) are obviously self-dual. With this notation, we have
\[
	\norm{v}^2 = \norm{v_+}^2 + \norm{v_-}^2 \geq 
	\norm{v_+}^2 - \norm{v_-}^2 = Q(v,v),
\]
with equality exactly when $Cv = v$. Among integral vectors with a fixed value of
$Q(v,v)$, those with $Cv = v$ therefore have the smallest possible Hodge norm
$\norm{v}^2$. In this setting, we have the following completely trivial finiteness
result: the set
\[
	H_q^+ = \menge{v \in \HZ}{\text{$Cv = v$ and $Q(v,v) = q$}}
\]
of self-dual integral vectors with a given self-intersection number $q \geq 1$ is
finite. Our main theorem is a generalization of this fact to arbitrary polarized
integral variations of Hodge structure of even weight.

\newpar
We now turn to the main result. Let $X$ be a nonsingular complex algebraic variety,
not necessarily complete, and let $\shH$ be a polarized integral variation of Hodge
structure on $X$, of even weight $2k$. Let $p \colon E \to X$ be the underlying
complex vector bundle, whose sheaf of holomorphic sections is isomorphic to $\OX
\tensor_{\ZZ} \shHZ$; here $\shHZ$ denotes the underlying local system of free
$\ZZ$-modules. At each point $x \in X$, the complex vector space $E_x = p^{-1}(x)$ is
equipped with a polarized integral Hodge structure of weight $2k$; the set of
integral vectors coincides with the stalk $\shHZx$. We denote by $C_x \in \End(E_x)$
the Weil operator, and by $Q_x$ the polarization; it is the stalk of the pairing $Q
\colon \shHZ \tensor_{\ZZ} \shHZ \to \ZZ_X$ that defines the polarization on $\shH$.
We shall think of the points of $E$ as pairs $(x,v)$, where $x \in X$ and $v \in
E_x$. 

Recall that $E$ is actually an algebraic vector bundle \cite{Deligne}; the algebraic structure is
uniquely determined by $\shH$. We shall give both $E$ and $X$ the
$\RRanexp$-definable structure extending their algebraic structure; then the
projection $p \colon E \to X$ becomes a morphism of definable spaces. Our main result
is that the set of all self-dual integral classes with fixed self-intersection number
is a \emph{definable} subspace of $E$.

\begin{theorem} \label{thm:main}
	Let $\shH$ be a polarized integral variation of Hodge structure of even weight
	on a nonsingular complex algebraic variety $X$. For each $q \geq 1$, the set
	\[
		\menge{(x,v) \in E}{\text{$v \in E_x$ is integral, $C_x v = v$, 
				and $Q_x(v,v) = q$}}
	\]
	is a definable, closed, real-analytic subspace of $E$, and the restriction of $p
	\colon E \to X$ to this set is proper with finite fibers.
\end{theorem}

\newpar
By analogy with the theorem of Cattani, Deligne, and Kaplan, it is a natural to ask
whether the locus of self-dual classes is actually real (semi-)algebraic, meaning
actually definable in the much smaller structure $\RRalg$. Simple examples in
dimension one show that semi-algebraic is the best one can hope for in general. We do
not know the answer to this question.

\newpar
Several useful variants of the main result can be obtained by tensoring with certain
auxiliary Hodge structures. The first one is the analogue of \cref{thm:main}
for integral classes that are anti-self-dual.

\begin{corollary} \label{cor:anti-self-dual}
	Let $\shH$ be a polarized integral variation of Hodge structure of even weight
	on a nonsingular complex algebraic variety $X$. For each $q \geq 1$, the set
	\[
		\menge{(x,v) \in E}{\text{$v \in E_x$ is integral, $C_x v = -v$, and
		$Q_x(v,v) = -q$}}
	\]
	is a definable, closed, real-analytic subspace of $E$, and the restriction of $p
	\colon E \to X$ to this set is proper with finite fibers.
\end{corollary}

The second one is a generalization to polarized integral variations of Hodge
structure of arbitrary weight, where we now consider pairs of integral classes that
are related by the Weil operator.

\begin{corollary} \label{cor:pairs}
	Let $\shH$ be a polarized integral variation of Hodge structure on a nonsingular
	complex algebraic variety $X$. For each $q \geq 1$, the set
	\[
		\menge{(x,v,w) \in E \times_X E}{\text{$v,w \in E_x$ are integral, $v = C_x w$, and
		$Q_x(v,w) = q$}}
	\]
	is a definable, closed, real-analytic subspace of $E \times_X E$, and the
	restriction of $p \colon E \times_X E \to X$ to this set is again proper with finite
	fibers.
\end{corollary}

Note that when $H$ is an integral Hodge structure of \emph{odd} weight, the Weil
operator $C \in \End(\HR)$ satisfies $C^2 = -\id$, and so its eigenvalues are the two
complex numbers $\pm i$. The condition $v = Cw$ is equivalent to
\[
	C(v + iw) = i(v + iw),
\]
which is saying that $v + iw \in \HZ \tensor_{\ZZ} \ZZ[i]$ is an eigenvector of the
Weil operator that is integral with respect to the Gaussian integers $\ZZ[i]$. 

\newpar
Unlike in the case of Hodge classes, the locus of self-dual (or anti-self-dual)
classes is in general not a complex analytic subset of the vector bundle $E$, hence
in particular not algebraic. The reason is that 
the Weil operator $C_x \in \End(E_x)$ depends real analytically -- but not complex
analytically -- on the point $x \in X$, which means that $C_x v = \pm v$
is not a holomorphic condition. We intend to discuss both the local structure of the
locus of self-dual classes, and its more precise behavior near a divisor with normal
crossing singularities, in a future paper. Here we only give two examples to show
what these loci can look like in practice.

\newpar
Our first example concerns anti-self-dual classes on  K3 surfaces; these show up naturally
in Verbitsky's study of ergodic complex structures on hyperk\"ahler manifolds
\cite{Verb1,Verb2}.

\begin{example}
	Let $S$ be a (not necessarily algebraic) K3 surface and let $\Lambda_\mathbb{Z}=
	H^2(S,\ZZ)$ together with the cup product pairing. The period domain $D$ parametrizing
	the Hodge structures of K3 surfaces is a complex manifold of dimension $20$;
	concretely, if $H^{2,0}(S) =
	\CC \sigma$, one has 
	\[
		D = \menge{[\sigma] \in \PP\Lambda_\mathbb{C}\cong \PP^{21}}{\text{$Q(\sigma, \sigma) = 0$ and
		$Q(\sigma, \bar{\sigma}) < 0$}}.
	\]
	The set of points where a given integral class $v \in \Lambda_\mathbb{Z}$ is
	anti-self-dual, of Hodge type $(2,0) + (0,2)$, is easily seen to be
	\[
		\menge{[\sigma] \in \PP^{21}}{
			Q(\sigma, \bar{\sigma}) v = Q(v, \bar{\sigma}) \sigma
		+ Q(v, \sigma) \bar{\sigma})}
	\]
	This is a totally real submanifold of real dimension $20$.
\end{example}
Using Ratner theory, Verbitsky shows that for any finite index subgroup
$\Gamma\subset\mathrm{O}(\Lambda_\ZZ)$, orbits $\Gamma p$ of elements $p \in
\Lambda_\mathbb{Z}$ come in three flavors:  closed orbits, dense orbits, and orbits whose closures are the $\Gamma$-orbit of an anti-self-dual locus.  The three behaviors correspond to the three possibilities $\rk ((H^{2,0}\oplus H^{0,2})\cap \Lambda_\mathbb{Z})=2,0,1$, respectively, for the Hodge structure associated to $p$.  The same analysis holds more generally for the period domain associated to the degree two cohomology of any (possibly singular) hyperk\"ahler variety, and is important for instance in the proof of the global Torelli theorem in the singular case \cite{BL}.  

\newpar
Our second example is self-dual classes in certain nilpotent orbits. This is less
geometric, but provides us with a large family of examples. For the general theory,
see \cite[\S3]{CKS} and the survey paper \cite[\S3]{CK} by Cattani and Kaplan.

\begin{example}
	Let $\HZ$ be a free $\ZZ$-module of finite rank, and let $Q \colon \HZ
	\tensor_{\ZZ} \HZ \to \ZZ$ be a nondegenerate symmetric bilinear pairing. Suppose
	that we have a representation $\rho \colon \sltwo \to \End(\HC)$ of the Lie
	algebra $\sltwo$, such that
	\[
		N = \rho \begin{pmatrix} 0 & 1 \\ 0 & 0 \end{pmatrix} \in \End(\HQ)
		\quad \text{and} \quad
		Y = \rho \begin{pmatrix} 1 & 0 \\ 0 & -1 \end{pmatrix} \in \End(\HR)
	\]
	satisfy $Q(Nv,w) + Q(v,Nw) = 0$ and $Q(Yv,w) = Q(v,Yw)$ for all $v,w \in \HC$, and
	such that $e^N \in \End(\HZ)$. Let $W_{\bullet}(N)$ denote the weight filtration
	of the nilpotent operator $N$. Let $F$ be a decreasing filtration of $\HC$
	such that
	\[
		Y(F^p) \subseteq F^p \quad \text{and} \quad N(F^p) \subseteq F^{p-1}
		\quad \text{for all $p \in \ZZ$.}
	\]
	Further assume that $\Fsh = e^{iN} F$ is the Hodge filtration of an integral Hodge
	structure of even weight on $\HZ$, polarized by the pairing $Q$. Then it is known
	\cite[Proposition~3.9]{CK} that the nilpotent orbit
	\[
		\HH \to D, \quad z \mapsto e^{zN} F,
	\]
	descends to a polarized integral variation of Hodge structure on $\dst$, whose
	monodromy transformation is $T = e^N$. Let us describe the
	locus of points in $\HH$ where a given integral class $v \in \HZ$ is self-dual.
	Write $z = x + iy$. Let $\Csh \in \End(\HR)$ denote the Weil operator of the Hodge
	structure $\Fsh$. From the identity
	\[
		e^{zN} F = e^{xN} e^{iyN} F 
		= e^{xN} e^{-\half \log y \, Y} e^{iN} e^{\half \log y \, Y} F
		= e^{xN} e^{-\half \log y \, Y} \Fsh
	\]
	and the fact that both exponential factors are elements of the real Lie group
	$G(\RR)$, it follows that the Weil operator of the Hodge structure $e^{zN} F$ is
	\[
		e^{xN} e^{-\half \log y \, Y} \Csh e^{\half \log y \, Y} e^{-xN}.
	\]
	The set of points $z \in \HH$ where our integral class $v \in \HZ$ is self-dual is
	therefore defined by the simple equation
	\begin{equation} \label{eq:Csh}
		\Csh \Bigl( e^{\half \log y \, Y} e^{-xN} v \Bigr)
		= e^{\half \log y \, Y} e^{-xN} v.
	\end{equation}
	At each point, the Hodge norm of $v$ is of course equal to $Q(v,v)$. If the set
	contains points with $y = \Im z$ arbitrarily large, then necessarily $v \in
	W_0(N)$. Since 
	\[
		W_0(N) = \bigoplus_{\ell \leq 0} E_{\ell}(Y),
	\]
	we have a decomposition $v = v_0 + v_{-1} + \dotsb$, where $Y v_{\ell} = \ell
	v_{\ell}$.  Now the Weil operator $\Csh$ interchanges the two weight spaces
	$E_{\pm \ell}(Y)$, because $Y \Csh + \Csh Y = 0$, for example by
	\cite[formulas~(6.35)]{CKS}. Since $e^{\half \log y \, Y}
	e^{-xN} v \in W_0(N)$, the identity in \eqref{eq:Csh} implies that $e^{\half \log
	y \, Y} e^{-xN} v \in E_0(Y)$, and hence that
	\[
		v = e^{xN} e^{-\half \log y \, Y} v_0 = e^{xN} v_0.
	\]
	Now there are two possibilities. Either $Nv = 0$ and $v = v_0$, or $Nv \neq 0$. In
	the first case, $v$ is self-dual at every point $z \in \HH$; in the second case,
	the equation $v = e^{xN} v_0$ uniquely determines the value of $x \in \RR$, and 
	$v$ is self-dual along the vertical ray $\Re z = x$. The connected components of
	the locus of self-dual classes are therefore of two different kinds: one kind
	projects isomorphically to the entire punctured disk $\dst$; the other to a single
	angular ray in $\dst$.
\end{example}

\def\R{\mathbb{R}}
\def\C{\mathbb{C}}
\def\Z{\mathbb{Z}}
\def\alg{\mathrm{alg}}
\section{Background on definability}
\newpar 
The theory of o-minimal structures provides a precise notion of tameness for subsets of euclidean space and functions on them.  It is flexible enough to allow for complicated constructions but restrictive enough to imply strong finiteness properties.
A general reference for this section is \cite{vdD}.

\newpar
To formalize this notion, we first introduce a way to describe collections of subsets of $\R^n$ which are closed under a variety of natural operations.

\begin{definition}\label{defn structure}
A structure $S$ is a collection $(S_n)_{n\in\mathbb{N}}$ where each $S_n$ is a set of subsets of $\R^n$ satisfying the following conditions:
\begin{enumerate}
\item Each $S_n$ is closed under finite intersections, unions, and complements;
\item The collection $(S_n)_{n\in\mathbb{N}}$ is closed under finite Cartesian products and coordinate projections;
\item For every polynomial $P\in\R[x_1,\ldots,x_n]$, the zero set 
\[(P=0):=\{x\in\R^n\mid P(x)=0\}\subset\R^n\]
 is an element of $S_n$.
\end{enumerate}
We refer to an element $U\in S_n$ as an $S$-definable subset $U\subset \R^n$.  For $U\in S_n,$ and $V\in S_m$, we say a map $f:U\to V$  of $S$-definable sets is $S$-definable if the graph is as a subset of $\R^{m+n}$.  When the structure $S$ is clear from context, we will often just refer to ``definable" sets and functions.
\end{definition}

\newpar For each $n$ taking $S_n$ to be the Boolean algebra generated by real algebraic subsets $(P=0)$ of $\R^n$, the resulting $S$ is not a structure.  For example, the algebraic set $(x^2-y=0)\subset \R^2$ projects to the semialgebraic set $(y\geq 0)\subset \R$.  On the other hand, if we take $S_n$ to be the Boolean algebra generated by real semialgebraic subsets $(P\geq 0)$ of $\R^n$, then by the Tarski--Seidenberg theorem the resulting $S=\R_\alg$ is a structure. 

\begin{note}
	Tarski--Seidenberg is usually phrased as quantifier elimination for the real ordered field.  Indeed, the above axioms for a structure say that definable sets are closed under first order formulas, as intersections, unions, and complements correspond to the logical operators ``and", ``or", and ``not", while the projection axiom corresponds to universal and existential quantifiers.  For this reason, structures have been studied extensively in model theory.
\end{note}

\newpar Surprisingly, a good notion of tame structure can be achieved by simply restricting the definable subsets of the real line.
\begin{definition}
	A structure $S$ is said to be o-minimal if $S_1=(\R_{\alg})_1$\textemdash that is, if the $S$-definable subsets of the real line are exactly finite unions of intervals.
\end{definition}

Sets and functions which are definable in an o-minimal structure have very nice properties, including the following.  Here we fix an o-minimal structure $S$ and by ``definable" we mean ``$S$-definable."
\begin{itemize}

\item For any definable function $f:U\to V$ with finite fibers, the fiber size $|f^{-1}(v)|$ is a definable function.  In particular, it is uniformly bounded, and for any $n$ the set
\[\{v\in V\mid |f^{-1}(v)|=n\}\]
is definable.

\item Any definable subset $U\subset\R^n$ admits a definable triangulation:  it is definably homeomorphic to a finite simplicial complex.  
\item Any definable subset $U\subset\R^n$ has a well-defined dimension, namely, the dimension as a simplicial complex for any definable triangulation.  Moreover, for any definable map $f:U\to V$ and any $n$ the set
\[\{v\in V\mid \dim f^{-1}(v)=n \}\]
is definable.
\item For any $k$ and any definable function $f:U\to \R$, there is a definable triangulation of $U$ such that $f$ is $C^k$ on each simplex.  As a consequence, any definable $U\subset\R^n$ can be partitioned into finitely many $C^k$-submanifolds of $\R^n$.

\end{itemize}

\newpar We give some examples of o-miminal structures. As discussed above, the structure $\R_{\alg}$ is o-minimal; in fact, it is the smallest o-minimal structure.  

Given a collection $\Sigma=(\Sigma_n)_{n\in\mathbb{N}}$ of subsets of $\R^n$ for each
$n$, we say the structure generated by $\Sigma$ is the smallest structure in which
each set in each $\Sigma_n$ is definable.  It is the structure whose definable sets
are given by first order formulas involving real polynomials, inequalities, and the
sets in $\Sigma$.  The structure $\R_{\exp}$ generated by the graph of the real
exponential $\exp:\R\to\R$ is o-minimal by a result of Wilkie
\cite{Wilkie:exponential}.  However, the function $\sin:\R\to\R$ is not definable in \textbf{any} o-minimal structure, as the definable set $\pi\Z=\sin^{-1}(0)$ is both discrete and infinite.

\newpar
To get a much larger o-minimal structure, 
let $\RRan$ be the structure generated by the graphs of all restrictions
$f|_{B(R)}$, where $f:B(R')\to\R$ is a real analytic functions on a finite radius
$R'<\infty$ open euclidean ball (centered at the origin) and $B(R)\subset B(R')$ is a
ball of strictly smaller radius $R<R'$.  Via the embedding $\R^n\subset
\R\mathbb{P}^n$, this is equivalent to the structure of subsets of $\R^n$ that are
subanalytic in $\R\mathbb{P}^n$.  As observed by van-den-Dries \cite{vdD},
Gabrielov's theorem of the complement implies that $\RRan$ is o-minimal.  Note
that while the sine function is not $\RRan$-definable, its restriction to any finite interval is.

Finally, let $\RRanexp$ be the structure generated by $\RRan$ and $\R_{\exp}$.  Then $\RRanexp$ is o-minimal by a result of van-den-Dries--Miller \cite{vdDM}.  Most of the applications to algebraic geometry currently use the structure $\RRanexp$, and this will be the default structure we work with.

\newpar

For applications, we typically wish to discuss definability for manifolds that don't arise as subsets of $\R^n$, and for this we need an appropriate notion of definable atlas. 

\begin{definition}
Let $M$ be a topological space and $S$ a structure.
\begin{itemize}
    \item A ($S$-)definable atlas $\{(U_i,\phi_i)\}$ consists of a finite open covering $U_i$ of $M$, and homeomorphisms $\phi_i:U_i\to V_i\subset\R^{n_i}$ such that
	\begin{enumerate}
	\item The $V_i$ and the pairwise intersections $V_{ij}:=\phi_i(U_i\cap U_j)$ are definable sets;
	\item The transition functions $\phi_{ij}:=\phi_j\circ \phi_i^{-1}:V_{ij}\to V_{ji}$ are definable.	
	\end{enumerate}
	\item If $M$ is equipped with a definable atlas $\{(U_i,\phi_i)\}$, we say a subset $Z\subset M$ is definable if each $\phi_i(U_i\cap Z)$ is.
	\item If $M,M'$ are equipped with definable atlases $\{(U_i,\phi_i)\}, \{(U'_{i'},\phi'_{i'})\}$, a map $f:M\to M'$ is definable if each $f^{-1}(U'_{i'})\subset M$ is definable and moreover for each $i,i'$ the composition
		\[(f\circ\phi_{i}^{-1})^{-1}(U'_{i'})\xrightarrow{\phi_i^{-1}}f^{-1}(U'_{i'})\xrightarrow{f}U'_{i'}\xrightarrow{\phi'_{i'}}V'_{i'}\]
	is ($S$-)definable.  	  
	\item We say two atlases $\{(U_i,\phi_i)\}, \{(U'_{i'},\phi'_{i'})\}$ on $M$ are equivalent if the identity map is definable with respect to $\{(U_i,\phi_i)\}$ on the source and $\{(U_{i'}',\phi'_{i'})\}$ on the target.
	\item A ($S$-)definable topological space is a topological space $M$ equipped with an equivalence class of definable atlases.  A morphism of ($S$-)definable topological spaces is a continuous map $f:M\to M'$ which is definable for any choice of atlases in the equivalence classes on the source and target.
	\item We likewise define ($S$-)definable manifolds (resp. ($S$-)definable complex manifolds) by in addition requiring that the charts map to open subsets of $\R^n$ (resp. $\C^n$) and that the transition functions are smooth (resp. holomorphic).  Here we make sense of definability in $\C^n$ via the identification $\C^n\cong \R^{2n}$ by taking real and imaginary parts. 
\end{itemize}

	
\end{definition}
\begin{example}
    Any complex algebraic variety $X$ naturally has the structure of an $S$-definable
	 topological space for any structure $S$.  It admits a finite covering $U_i$ by
	 affine algebraic varieties, each of which is a complex (hence real) algebraic
	 subset of some $\C^{n_i}$, and the transition functions are given by algebraic
	 functions.  Likewise, any nonsingular complex algebraic variety has a natural
	 structure as an $S$-definable complex manifold.  We denote this definable complex manifold by $X^\mathrm{def}$.
\end{example}

\newpar One common method of producing interesting definable topological spaces is by taking quotients of other definable topological spaces by definable group actions.

\begin{definition}Let $X$ be a locally compact Hausdorff definable topological space and $\Gamma$ a group acting on $X$ by definable homeomorphisms.  A definable fundamental set for the action of $\Gamma$ on $X$ is an open definable subset $F \subseteq X$ such that
\begin{enumerate}
\item $\Gamma \cdot F = X$,
\item the set $\{ \gamma \in \Gamma \, | \, \gamma \cdot F  \cap F \neq \varnothing \}$ is finite.
\end{enumerate}
\end{definition}

We shall require the following proposition:

\begin{proposition}[\emph{cf.}~{\cite[Proposition~2.3]{bbkt}}]
If $F$ is a definable fundamental set for the action of $\Gamma$ on $X$, then there exists a unique definable structure on $\Gamma \backslash X$ such that the canonical map $F \rightarrow \Gamma \backslash X$ is definable.
\end{proposition}

\section{Background on Siegel sets}

\newpar
In this section, we give some background on Siegel sets for symmetric spaces of the
type that appear in the study of period mappings. We begin by recalling the
definition of a Siegel set; a good general discussion is \cite[\S\,2]{BJ}. Let $G$ be a
reductive $\QQ$-algebraic group. The set of real points $G(\RR)$ is a real Lie group,
and we fix a maximal compact subgroup $K \subseteq G(\RR)$. By \cite[Proposition~1.6]{BS},
this determines a Cartan involution $\theta$ of the $\RR$-algebraic group $G_{\RR}$,
whose fixed point set is $K$. Let $P \subseteq G$ be a minimal parabolic
$\QQ$-subgroup, and let $U \subseteq
P$ be its unipotent radical. The dimension of any maximal split $\QQ$-torus in $P$ is
called the $\QQ$-rank of $G$; we shall denote it by $\rk G$. The Levi quotient
$P/U$ is isomorphic to the product of a split $\QQ$-torus of dimension $\rk G$ and a
maximal anisotropic $\QQ$-subgroup.  According to \cite[Corollary~1.9]{BS}, there is a
unique Levi subgroup 
\[
	L = S \times M \subseteq P_{\RR}
\]
that maps isomorphically to the Levi quotient $P_{\RR} / U_{\RR}$ and is stable under
the Cartan involution $\theta$. In particular, $S$ is a split $\RR$-torus of
dimension $\rk G$ that is conjugate over $G_{\RR}$ to a maximal $\QQ$-split torus of
$G$, such that $\theta(g) = g^{-1}$ for every $g \in S(\RR)$; compare
\cite[Lemma~2.1]{Orr}. Moreover, $M$ is contained in the centralizer of $S$ in
$G_{\RR}$. The adjoint action of $S$ on the Lie algebra of $P_{\RR}$ determines a
root system, and we write $\Delta$ for the set of simple roots.

We use the definition of Siegel sets in \cite[\S\,2.2]{Orr}; for a discussion of how it
relates to the original definition in \cite[\S\,12]{Borel:IntroductionArithmeticGroups}, see \cite[\S\,2.3]{Orr}. 
For our purposes, a \emph{Siegel set} in $G(\RR)$, with respect to the maximal compact
subgroup $K$ and the minimal parabolic $\QQ$-subgroup $P$, is any set of the form
\[
	\Sieg(\Omega, t) = \Omega \cdot A_t \cdot K \subseteq G(\RR), 
\]
where $\Omega \subseteq U(\RR) M(\RR)^+$ is a compact set, $t > 0$
is a positive real number, and
\[
	A_t = \menge{g \in S(\RR)^+}{\text{$\chi(g) \geq t$ 
			for all simple roots $\chi \in \Delta$}}.
\]
We say that a Siegel set is \emph{subalgebraic} if $\Omega \subseteq U(\RR) M(\RR)^+$
is subalgebraic, meaning definable in the o-minimal structure $\RRalg$. 
\begin{note}
	More generally, it is known that any set of the form $\Sieg(\Omega, t)$ with
	$\Omega \subseteq P(\RR)$ compact is contained in a Siegel set in the above sense
	\cite[\S\,2.3]{Orr}.
\end{note}

\newpar
We are only going to be interested in Siegel sets with respect to a fixed maximal
compact subgroup. To emphasize this, we usually talk about \define{Siegel sets for
$K$}, meaning that the maximal compact subgroup $K$ in the definition is fixed,
whereas the minimal parabolic $\QQ$-subgroup $P$ is allowed to be arbitrary. For
later use, let us briefly recall how Siegel sets for different minimal parabolic
$\QQ$-subgroups are related to each other. Let $K \subseteq G(\RR)$ be a fixed
maximal compact subgroup, and $P \subseteq G$ be a minimal parabolic $\QQ$-subgroup.
Since all minimal parabolic $\QQ$-subgroups are conjugate to each other, any other
choice $P' \subseteq G$ has the form $P' = g P g^{-1}$ for a suitable element $g \in
G(\QQ)$.  Write $g = kp$ with $k \in K$ and $p \in P(\RR)^+$, so that
\[
	P'(\RR) = g P(\RR) g^{-1} = k P(\RR) k^{-1}.
\]
Now suppose that $\Sieg \subseteq G(\RR)$ is a Siegel set for $K$ and $P$. By
\cite[\S\,12.4]{Borel:IntroductionArithmeticGroups}, the translate $p \Sieg$ is contained in a larger
Siegel set $\Omega A_t K \subseteq G(\RR)$ with respect to $K$ and $P$; here $\Omega
\subseteq U(\RR) M(\RR)^+$ is compact and $t > 0$. Consequently,
\[
	g \Sieg \subseteq k \Omega A_t K = 
	k \Omega k^{-1} \cdot k A_t k^{-1} \cdot K,
\]
and the right-hand side is now a Siegel set with respect to $K$ and $P' = g P
g^{-1}$. This shows that any Siegel set for $K$ is contained in a $G(\QQ)$-translate
of a Siegel set with respect to $K$ and a fixed minimal parabolic $\QQ$-subgroup $P$.

\newpar \label{par:Hodge-G}
Now we specialize to the case that is of interest in the study of period
mappings. The general setting is as follows. Let $\HQ$ be a finite-dimensional
$\QQ$-vector space, of dimension $n = \dim \HQ$, equipped with a nondegenerate
symmetric bilinear form
\[
	Q \colon \HQ \tensor_{\QQ} \HQ \to \QQ.
\]
Further suppose that there is an endomorphism $C \in \End(\HR)$ of the real vector
space $\HR = \HQ \tensor_{\QQ} \RR$ that satisfies $C^2 = \id$, such that
\[
	\inner{\argbl}{\argbl}_C \colon \HR \tensor_{\RR} \HR \to \RR, \quad
	\inner{v}{w} = Q(v, Cw),
\]
is a positive definite inner product on $\HR$. By analogy with the case of Hodge
structures, we shall say that $C$ is a \define{Weil operator} for the pair $(\HQ, Q)$.

\newpar
The orthogonal group $G = O(\HQ, Q)$ is a reductive $\QQ$-algebraic group, in general
not connected, whose set of real points is the real Lie group
\[
	G(\RR) = \menge{g \in \Aut(\HR)}{\text{$Q(g v, g w) = Q(v,w)$ for all $v,w \in \HR$}}.
\]
Evidently, $C \in G(\RR)$. It is easy to see that an element $g \in
G(\RR)$ preserves the inner product $\inner{\argbl}{\argbl}_C$ if and only if $gC =
Cg$; therefore the subgroup
\[
	K = \menge{g \in G(\RR)}{gC = Cg}
\]
is compact. It is proved in \cite[(8.4)]{Schmid} that $K$ is actually a maximal
compact subgroup of $G(\RR)$, and that the associated Cartan involution is
given by the simple formula
\[
	\theta \colon G(\RR) \to G(\RR), \quad \theta(g) = C g C.
\]
The following result is well-known.

\begin{lemma} \label{lem:Weil}
	The symmetric space $G(\RR)/K$ parametrizes Weil operators for $(\HQ,
	Q)$, with the coset $gK$ corresponding to the Weil operator $gCg^{-1} \in \End(\HR)$. 
\end{lemma}

\begin{proof}
	All elements in the coset $gK$ give us the same operator $gCg^{-1} \in
	\End(\HR)$, which is a Weil operator for the pair $(\HQ, Q)$ because
	\[
		Q(v, g C g^{-1} w) = Q(g^{-1} v, C g^{-1} w) = \inner{g^{-1} v}{g^{-1} w}_C
	\]
	is positive definite. Conversely, suppose that $C' \in \End(\HR)$ is another Weil
	operator for $(\HQ, Q)$. Let $n = \dim \HR$. Since $Q$ has a fixed signature,
	we have
	\[
		\dim E_1(C') = \dim E_1(C) = p \quad \text{and} \quad
		\dim E_{-1}(C') = \dim E_{-1}(C) = n-p.
	\]
	Pick a basis $e_1, \dotsc, e_n \in \HR$ that is orthonormal for the inner product
	$\inner{\argbl}{\argbl}_C$, in such a way that $e_1, \dotsc, e_p \in E_1(C)$ and
	$e_{p+1}, \dotsc, e_n \in E_{-1}(C)$. Pick a second basis $e_1',
	\dotsc, e_n' \in \HR$ that is similarly adapted to $\inner{\argbl}{\argbl}_{C'}$
	and $C'$, and let $g \in \Aut(\HR)$ be the unique automorphism such that $g e_i =
	e_i'$ for $i = 1, \dotsc, n$. Then
	\[
		Q(e_i, C e_j) 
		= Q(e_i', C' e_j') = Q(g e_i, C' g e_j) = Q(g e_i, g C e_j),
	\]
	and therefore $g \in G(\RR)$ by the nondegeneracy of $Q$. By construction, 
	$C' g = gC$, which makes $C' = g C g^{-1}$ equal to the image of the coset $gK$.
\end{proof}

\newpar
Now let us turn our attention to Siegel sets in $G(\RR)$. The $\QQ$-rank of $G$ and
the collection of minimal parabolic $\QQ$-subgroups $P \subseteq G$ can be described
concretely as follows. Let $r \geq 0$ be the Witt rank of $Q$, meaning the dimension
of a maximal $Q$-isotropic subspace of $\HQ$. Let
\[
	\{0\} \subset V_1 \subset V_2 \subset \dotsb \subset V_r
\]
be a maximal flag of isotropic subspaces, with $\dim V_i = i$. As explained in
\cite[\S\,11.16]{Borel:IntroductionArithmeticGroups}, the stabilizer $P$ of this flag is a minimal parabolic
$\QQ$-subgroup of $G$, and every minimal parabolic $\QQ$-subgroup arises in this way;
moreover, the $\QQ$-rank of $G$ is equal to $r$. Since $Q$ is nondegenerate, it is
possible to choose vectors $v_1', \dotsc, v_r' \in \HQ$ with the property that
\[
	Q(v_i, v_j') = [i = j] = \begin{cases}
		1 & \text{if $i=j$,} \\
		0 & \text{otherwise.}
	\end{cases}
\]
The $2r$-dimensional subspace spanned by $v_1, \dotsc, v_r, v_1', \dotsc, v_r'$ is
uniquely determined by $V_r$; so is its orthogonal complement with respect to $Q$.
Let $U \subseteq P$ denote the unipotent radical; concretely, $g \in U(\QQ)$ iff
$g v_i - v_i \in V_{i-1}$ for $i = 1, \dotsc, r$.

\newpar
The unique Levi subgroup $S \times M \subseteq P_{\RR}$ that is stable under the
Cartan involution can be described concretely as follows. Using the Gram--Schmidt
process, construct an orthonormal basis $e_1, \dotsc, e_r \in V_r$ relative to the
inner product $\inner{\argbl}{\argbl}_C$, in such a way that
\[
	V_i \tensor_{\QQ} \RR = \RR e_1 \oplus \dotsb \oplus \RR e_i
\]
for $i = 1, \dotsc, r$. Since $V_r$ is isotropic, the vectors $e_1, \dotsc, e_r, C
e_r, \dotsc, C e_1$ are still orthonormal, and we get an embedding
\[
	s \colon G_{m, \RR} \times \dotsb \times G_{m, \RR} \into P_{\RR}
\]
by letting $s(\lambda_1, \dotsc, \lambda_r)$ act as multiplication by
$\lambda_i$ on the vector $e_i$, as multiplication by $\lambda_i^{-1}$ on the vector
$C e_i$, and as the identity on the orthogonal complement of $e_1, \dotsc, e_r, C
e_r, \dotsc, C e_1$. The image of this embedding is the desired $\RR$-torus $S$. 
The other factor of the Levi subgroup $S \times M$ has as its set of real points
\[
	M(\RR) = \menge{g \in G(\RR)}{\text{$g e_i = CgC e_i = e_i$ for all $i=1, \dotsc,
	r$}},
\]
which is clearly stable under the Cartan involution $\theta(g) = C g C$. Note in
particular that $M(\RR)$ preserves the orthogonal complement of $e_1, \dotsc, e_r, C
e_r \dotsc, C e_1$.

\newpar
We also need to know the set of simple roots $\Delta$ for the action of $S$ on the
Lie algebra of $P_{\RR}$. These are computed in \cite[\S\,11.16]{Borel:IntroductionArithmeticGroups}. There are two
cases, depending on the value of the integer $n-2r \geq 0$:
\begin{enumerate}
	\item If $n = 2r$, the simple roots are $\lambda_1/\lambda_2, \dotsc,
		\lambda_{r-1}/\lambda_r$ and $\lambda_{r-1} \lambda_r$; this is the case where
		$(\HQ, Q)$ is split, hence isomorphic to a sum of hyperbolic planes.
	\item If $n > 2r$, the simple roots are $\lambda_1/\lambda_2, \dotsc,
		\lambda_{r-1}/\lambda_r$ and $\lambda_r$; this is the case where
		$(\HQ, Q)$ has a nontrivial anisotropic summand.
\end{enumerate}

\newpar
As in the case of the general linear group, Siegel sets in $G(\RR)$ are closely
related to the reduction theory of positive definite quadratic forms. We are going to
make this idea precise by comparing Siegel sets for the two $\QQ$-algebraic groups $G
= O(\HQ, Q)$ and $\Gtl = \GL(\HQ)$. In $\Gtl$, we have the maximal compact subgroup
\[
	\Ktl = \menge{g \in \Gtl(\RR)}{\text{$\inner{g^{-1} v}{g^{-1} w}_C = \inner{v}{w}_C$
	for $v,w \in \HR$}}.
\]
The associated Cartan involution is $g \mapsto C (g^t)^{-1} C$, where $g^t$ means the
adjoint of $g$ with respect to the nondegenerate pairing $Q$. Clearly, $\Ktl \cap
G(\RR) = K$. The relevant minimal parabolic $\QQ$-subgroup $\Ptl \subseteq \Gtl$ is
obtained as follows. Complete the given flag $V_1 \subset \dotsb \subset
V_r$ of isotropic subspaces to a maximal flag
\[
	\{0\} \subset V_1 \subset \dotsb \subset V_r \subset \dotsb \subset V_{n-r}
	\subset V_{n-r+1} \subset \dotsb \subset V_n = \HQ
\]
by defining $V_{n-r} = V_r^{\perp}$ and $V_{n-r+i} = V_r^{\perp} \oplus \QQ v_r'
\oplus \dotsb \oplus \QQ v_{r+1-i}'$ for $i = 1, \dotsc, r$, and then filling in the
$n - 2r$ steps in between $V_r$ and $V_r^{\perp}$. Let $\Ptl \subseteq \Gtl$ be the
stabilizer of this maximal flag, and let $\Utl \subseteq \Ptl$ be its unipotent
radical; then 
\[
	\Ptl \cap G \subseteq P \quad \text{and} \quad \Utl \cap G = U.
\]
Using the Gram-Schmidt process, construct an orthonormal basis $e_1,
\dotsc, e_n \in \HR$ with the property that $V_i \tensor_{\QQ} \RR = \RR e_1 \oplus
\dotsb \oplus \RR e_i$; a short calculation shows that
\[
	\text{$e_{n-r+1} = C e_r$, $e_{n-r+2} = C e_{r-1}$, \dots, $e_n = C e_1$.}
\]
In this case, the Levi subgroup $\Stl \times \Mtl$ reduces to the split $\RR$-torus
$\Stl \subseteq \Ptl_{\RR}$ consisting of all diagonal matrices $\diag(\lambda_1,
\dotsc, \lambda_n)$ with respect to the basis $e_1, \dotsc, e_n$; with a little bit
of work, one can show that $\Stl \cap G(\RR) = S$. The simple roots are computed in
\cite[\S\,1.14]{Borel:IntroductionArithmeticGroups} to be $\lambda_1/\lambda_2, \dotsc, \lambda_{n-1}/\lambda_n$. 

\newpar
We can now compare Siegel sets in $G(\RR) = O(\HR, Q)$ and $\Gtl(\RR) =
\GL(\HR)$.  The result is a more precise version of a general theorem by Orr
\cite[Theorem~1.2]{Orr}, with a small correction contained in \cite{OS}.

\begin{proposition} \label{prop:Siegel}
	Any Siegel set in $G(\RR)$ for the maximal compact subgroup $K$ is contained in
	at most two $G(\QQ)$-translates of a Siegel set in $\Gtl(\RR)$ $($for $\Ktl)$.
\end{proposition}

\begin{proof}
	We need to consider the two cases $n > 2r$ and $n = 2r$ separately. Let us first
	deal with the easier case $n > 2r$ (where $\HQ$ has a nontrivial anisotropic
	summand). Without loss of generality, we can assume that the Siegel set in
	$G(\RR)$ has the form
	\[
		\Omega_U \cdot \Omega_M \cdot A_t \cdot K,
	\]
	where $\Omega_U \subseteq U(\RR)$ and $\Omega_M \subseteq M(\RR)^+$ are compact
	subsets and $t > 0$ is a real number. From the description of the simple
	roots above, we know that
	\[
		A_t = \menge{s(\lambda_1, \dotsc, \lambda_r)}{\text{$\lambda_1/\lambda_2 \geq
		t$, \dots, $\lambda_{r-1}/\lambda_r \geq t$, and $\lambda_r \geq t$}}.
	\]
	It will be convenient to write elements of $\Gtl(\RR)$ as matrices with respect to
	our fixed orthonormal basis $e_1, \dotsc, e_n \in \HR$. With this convention, the
	set $A_t$ consists of all diagonal matrices of the form
	\[
		\diag(\lambda_1, \dotsc, \lambda_r, 1, \dotsc, 1, \lambda_r^{-1}, \dotsc,
		\lambda_1^{-1})
	\]
	with $\lambda_1/\lambda_2 \geq t$, \dots, $\lambda_{r-1}/\lambda_r \geq t$ and
	$\lambda_r \geq t$. The crucial point is that every such matrix belongs to
	$\Atl_t \subseteq \Stl(\RR)$, because the number of $1$'s in the middle is $n-2r
	\geq 1$. Consider an arbitrary element
	\[
		u \cdot m \cdot a \cdot k \in \Omega_U \cdot \Omega_M \cdot A_t \cdot K.
	\]
	As a matrix, $u$ is upper triangular with all diagonal elements equal to $1$, and
	$m$ is block-diagonal, with the first $r$ and last $r$ diagonal entries equal to
	$1$; in particular, we have $ma = am$.  Since $m \in \Omega_M$ varies in a compact
	set, the components of the polar decomposition
	\[
		m = \ptl_m \cdot \ktl_m \in \Ptl(\RR) \cdot \Ktl
	\]
	also belong to compact subsets of $\Ptl(\RR)$ and $\Ktl$; moreover, $\ptl_m$ is again
	block-diagonal, and therefore $a \ptl_m = \ptl_m a$. This gives
	\[
		u \cdot m \cdot a \cdot k = u \ptl_m \cdot a \cdot \ktl_m k 
		\in \Ptl(\RR) \cdot \Atl_t \cdot \Ktl.
	\]
	The first factor lies in a compact subset of $\Ptl(\RR)$, and we have already noted
	that $a \in \Atl_t$; consequently, our Siegel set is contained in a Siegel set in
	$\Gtl(\RR)$.

	The split case $n = 2r$ is less straightforward. Here the subgroup $M$ is
	trivial, and therefore our Siegel set takes the form
	\[
		\Sieg(\Omega, t) = \Omega \cdot A_t \cdot K \subseteq G(\RR),
	\]
	where $\Omega \subseteq U(\RR)$ is compact. The simple roots for the action of
	$S$ on the Lie algebra of $P_{\RR}$ are now $\lambda_1/\lambda_2, \dotsc,
	\lambda_{r-1}/\lambda_r$, and $\lambda_{r-1} \lambda_r$; with respect to our
	orthonormal basis $e_1, \dotsc, e_{2r} \in \HR$, the set $A_t$ consists of
	all diagonal matrices of the form
	\[
		\begin{pmatrix}
			\lambda_1 & & & \\
			& \ddots & & \\
			& & \lambda_r \\
			& & & \lambda_r^{-1} & & & \\
			& & & & \ddots & & \\
			& & & & & \lambda_1^{-1}
		\end{pmatrix}
	\]
	with $\lambda_1/\lambda_2 \geq t$, \dots, $\lambda_{r-1}/\lambda_r \geq t$, and
	$\lambda_{r-1} \lambda_r \geq t$. As long as $\lambda_r \geq 1$, this matrix
	belongs to $\Atl_{\min(1,t)}$; but if $\lambda_r \leq 1$, this only holds after we
	swap $\lambda_r$ and $\lambda_r^{-1}$, which amounts to conjugating by the
	permutation matrix
	\[
		\sigma = \begin{pmatrix}
			1 & & & \\
			& \ddots & & \\
			& & 0 & 1 \\
			& & 1 & 0 & & & \\
			& & & & \ddots & & \\
			& & & & & 1 
		\end{pmatrix} \in \Ktl.
	\]
	In a nutshell, this is the reason why we need \emph{two} translates of a Siegel
	set. Getting down to the details, consider again an arbitrary element
	\[
		u \cdot a \cdot k \in \Omega \cdot A_t \cdot K.
	\]
	If $a \in A_t$ is such that $\lambda_r \geq 1$, then $a \in \Atl_{\min(1,t)}$, and
	we can argue as before to show that this part of $\Sieg(\Omega, t)$ is contained in a
	Siegel set in $\Gtl(\RR)$. Let us therefore suppose that $\lambda_r \leq 1$. We
	can rewrite our element in the form
	\[
		u \cdot a \cdot k 
		= \sigma \cdot u^{\sigma} \cdot a^{\sigma} \cdot \sigma k
		\in \sigma \cdot u^{\sigma} \cdot \Atl_{\min(1,t)} \cdot \Ktl,
	\]
	where $a^{\sigma} = \sigma a \sigma$ etc. Now the crucial point is that
	$u^{\sigma} \in \Utl(\RR)$, which puts this part of $\Sieg(\Omega, t)$ into the
	translate by $\sigma$ of a Siegel set in $\Gtl(\RR)$. 

	Here is the reason why $u^{\sigma} = \sigma u \sigma \in \Utl(\RR)$.
	The matrix for $u$ is upper triangular with all diagonal
	entries equal to $1$; in particular, there is some $x \in \RR$ such that
	\[
		u e_{r+1} \equiv e_{r+1} + x e_r \mod \langle e_1, \dotsc, e_{r-1} \rangle.
	\]
	Because of the special shape of $\sigma$, having $u^{\sigma} \in \Utl(\RR)$ is
	now equivalent to
	\[
		x = \inner{u e_{r+1}}{e_r}_C = Q(u e_{r+1}, C e_r) = Q(u e_{r+1}, e_{r+1}) = 0.
	\]
	As $u \in G(\RR)$, we have $Q(u e_{r+1}, e_{r+1}) = Q(u^{-1} e_{r+1},
	e_{r+1})$, and therefore
	\[
		\inner{u e_{r+1}}{e_r}_C = \inner{u^{-1} e_{r+1}}{e_r}_C.
	\]
	From the relation $u e_{r+1} \equiv e_{r+1} + x e_r$, we deduce that
	\[
		u^{-1} e_{r+1} \equiv e_{r+1} - x e_r 
			\mod \langle e_1, \dotsc, e_{r-1} \rangle,
	\]
	and after taking the inner product with $e_r$, we get $x = -x$ or $x=0$. 

	It remains to argue that the translate is actually by an element of $G(\QQ)$; note
	that $\sigma \in \Ktl$ is not rational in general. To that end, we define an
	involution $g \in G(\QQ)$ by requiring that $g v_i = v_i$ and $g
	v_i' = v_i'$ for $i = 1, \dotsc, r-1$, and $g v_r = v_r'$ and $g v_r' = v_r$. Then
	it is easy to check that the matrix for $g \sigma$ in the basis $e_1, \dotsc,
	e_{2r}$ is upper triangular, which means that $g \sigma \in \Ptl(\RR)$. Using
	\cite[\S\,12.4]{Borel:IntroductionArithmeticGroups}, it follows that the translate by $\sigma$ of a Siegel set in
	$\Gtl(\RR)$ with respect to $\Ktl$ and $\Ptl$ is contained in the translate by $g$ of a
	bigger Siegel set in $\Gtl(\RR)$, which is enough for our purposes.
\end{proof}

\newpar
Let us now relate Siegel sets in $G(\RR)$ to reduction theory for quadratic forms.
Following \cite[Section~I.2]{Klingen}, we shall say that a positive definite inner product
$\inner{\argbl}{\argbl}$ on the vector space $\HR$ is \define{$t$-reduced} relative
to an ordered basis $v_1, \dotsc, v_n \in \HQ$ (where $t > 0$ is a real number) if the
following three conditions hold:
\begin{aenumerate}
\item For every $1 \leq i \leq n-1$, one has $\norm{v_i}^2 \leq t \norm{v_{i+1}}^2$.
		\item For every $1 \leq i < j \leq n$, one has $2 \abs{\inner{v_i}{v_j}} \leq t
		\norm{v_i}^2$.
	\item The matrix of the quadratic form satisfies the inequality
		\[
			\prod_{i=1}^n \norm{v_i}^2 \leq t \cdot c_1(n) \det \bigl( \inner{v_i}{v_j}
			\bigr)_{i,j},
		\]
		where $c_1(n)$ the optimal constant in Minkowski's inequality.
\end{aenumerate}
For a given basis $v_1, \dotsc, v_n \in \HQ$ and a given number $t > 0$, consider the
set of elements $g \in \GL(\HR)$ such that the inner product
\[
	(v, w) \mapsto \inner{g^{-1} v}{g^{-1} w}_C = Q(g^{-1} v, C g^{-1} w)
\]
is $t$-reduced relative to the basis $v_1, \dotsc, v_n$. It is known that every
Siegel set in $GL(\HR)$ for the maximal compact subgroup $\Ktl$ is contained in a
set of this type; conversely, every set of this type is contained in a Siegel set
(for $\Ktl$).

\newpar
To simplify the discussion, let us denote by
\[
	\Sieg(v_1, \dotsc, v_n, t)
\]
the set of elements $g \in G(\RR)$ such that the inner product
\[
	(v, w) \mapsto \inner{v}{w}_{g C g^{-1}} = Q(v, g C g^{-1} w)
\]
is $t$-reduced relative to a given basis $v_1, \dotsc, v_n \in \HQ$. Being defined by a
collection of inequalities, this is clearly a subalgebraic subset of $G(\RR)$. It is
easy to see that
\[
	g \Sieg(v_1, \dotsc, v_n, t) = \Sieg(g v_1, \dotsc, g v_n, t)
\]
for any $g \in G(\QQ)$, which means that the collection of these sets is stable under
translation by elements in $G(\QQ)$. The following theorem gives a useful
criterion for checking whether a given subset of $G(\RR)$ is contained in finitely
many $G(\QQ)$-translates of a Siegel set for $K$.

\begin{theorem} \label{thm:Siegel-reduction}
	Any Siegel set in $G(\RR)$ for the maximal compact subgroup $K$ is contained in
	a finite union of sets of the form $\Sieg(v_1, \dotsc, v_n, t)$; conversely, any
	set of the form $\Sieg(v_1, \dotsc, v_n, t)$ is contained in a finite union of
	Siegel sets $($for $K)$.
\end{theorem}

\begin{proof}
	The first assertion follows immediately from \cref{prop:Siegel} and the
	discussion above. To prove the second assertion, observe that any set of the form
	$\Sieg(v_1, \dotsc, v_n, t)$ is contained in a Siegel set in $\GL(\HR)$ (for
	$\Ktl$); moreover, the intersection of such a Siegel set with the subgroup
	$G(\RR)$ is contained in finitely many $G(\QQ)$-translates of a Siegel set in
	$G(\RR)$ by \cref{prop:BHC} below. We then get the desired result by
	recalling that any $G(\QQ)$-translate of a Siegel set for $K$ is again a 
	Siegel set for $K$ (with a possibly different minimal parabolic $\QQ$-subgroup).
\end{proof}

\newpar
In this section, we prove a proposition concerning the intersection of a Siegel set with
a reductive subgroup. The result is similar to \cite[Lemma~7.5]{BHC}, except that we
are working with Siegel sets for $\QQ$-algebraic groups (instead of with Siegel
domains for $\RR$-algebraic groups), and that we are making a different set of
assumptions about the subgroup. To simplify the notation, let $G$ be an
arbitrary reductive $\QQ$-algebraic group, and $H \subseteq G$ a reductive
$\QQ$-algebraic subgroup. Further, let $K_G \subseteq G(\RR)$ and $K_H \subseteq
H(\RR)$ be maximal compact subgroups such that $K_H = K_G \cap H(\RR)$.
Note that the Cartan involutions on $G_{\RR}$ and $H_{\RR}$ are not necessarily
compatible with each other.

\begin{proposition} \label{prop:BHC}
	Suppose that every Siegel set in $H(\RR)$ for the maximal compact subgroup $K_H$
	is contained in finitely many $G(\QQ)$-translates of a Siegel set in $G(\RR)$ (for
	$K_G$). Let $\Sieg_G \subseteq G(\RR)$ be any Siegel set for $K_G$. Then there is a
	Siegel set $\Sieg_H \subseteq H(\RR)$ for $K_H$, and a finite set $F \subseteq
	H(\QQ)$, such that
	\[
		\Sieg_G \cap H(\RR) \subseteq F \cdot \Sieg_H.
	\]
\end{proposition}

\begin{proof}
	This is an easy consequence of reduction theory, and all that is required is
	collecting some results from \cite{Borel:IntroductionArithmeticGroups}. Let $\Gamma_G \subseteq G(\QQ)$ be an
	arithmetic subgroup; the intersection $\Gamma_H = \Gamma_G \cap H(\QQ)$ is then an
	arithmetic subgroup of $H(\QQ)$. According to \cite[Theorem~15.5]{Borel:IntroductionArithmeticGroups}, there
	exists a Siegel set $\Sieg_H \subseteq H(\RR)$ for the maximal compact subgroup
	$K_H$, and a finite set $C_H \subseteq H(\QQ)$, such that
	\[
		H(\RR) = \Gamma_H \cdot C_H \cdot \Sieg_H.
	\]
	Since the intersection $\Sieg_G \cap H(\RR)$ is of course contained in $\Gamma_H
	C \Sieg_H$, it is therefore enough to prove finiteness of the set
	\[
		B = \menge{\gamma \in \Gamma_H}{\text{$\Sieg_G$ intersects $\gamma C_H
		\Sieg_H$}}.
	\]
	After enlarging $\Sieg_G$, if necessary, our assumption about Siegel sets in
	$H(\RR)$ implies that there is a finite set $A \subseteq G(\QQ)$ such that $\Sieg_H
	\subseteq A \Sieg_G$. Consequently, our set $B \subseteq \Gamma_H$ is
	contained in the larger set
	\[
		\menge{\gamma \in \Gamma_G}{\text{$\Sieg_G$ intersects $\gamma C_H A \Sieg_G$}},
	\]
	which is finite because $\Sieg_G$ has the Siegel property \cite[Theorem~15.4]{Borel:IntroductionArithmeticGroups}.

	It remains to justify our claim that $\Sieg_H \subseteq A \Sieg_G$. By assumption,
	$\Sieg_H$ is contained in many $G(\QQ)$-translates of a Siegel set in $G(\RR)$,
	but probably with respect to a different minimal parabolic $\QQ$-subgroup. After
	translation by an element of $G(\QQ)$, we can assume that the minimal parabolic
	$\QQ$-subgroup is the same as for $\Sieg_G$; and then we can enlarge $\Sieg_G$ and
	assume that the Siegel set in question is actually $\Sieg_G$. This completes the proof.
\end{proof}

\newpar
Let us return to the setting considered in \cref{par:Hodge-G}, but add one
additional piece of data. We still assume that $G = O(\HQ, Q)$, and that we have a
fixed Weil operator $C \in \End(\HR)$ for which $(v,w) \mapsto \inner{v}{w}_C =
Q(v,Cw)$ is a positive definite inner product on the $\RR$-vector space $\HR$. Let us
now assume in addition that we have a nonzero element $a \in \HQ$ with the property
that $Ca = a$. In particular,
\[
	Q(a,a) = \inner{a}{a}_C > 0.
\]
The stabilizer of this vector is a reductive $\QQ$-subgroup $G_a \subseteq G$; concretely,
\[
	G_a(\QQ) = \menge{g \in G(\QQ)}{ga = a}.
\]
We are interested in Weil operators $C' \in \End(\HR)$ with the property that $C' a =
a$. The following result says that all such Weil operators are conjugate to $C$ by
elements of the real group $G_a(\RR)$.

\begin{lemma} \label{lem:orbit-Ga}
	Let $C' \in \End(\HR)$ be a Weil operator for $(\HQ, Q)$. If $C' a = a$, then
	there is an element $g \in G_a(\RR)$ such that $C' = g C g^{-1}$.
\end{lemma}

\begin{proof}
	The proof is similar to that of \cref{lem:Weil}. We have
	\[
		\norm{a}_{C'}^2 = Q(a,a) = \norm{a}_C^2,
	\]
	because $C'a = a = Ca$. Let $n = \dim \HR$. Since $Q$ has a fixed signature, we have
	\[
		\dim E_1(C') = \dim E_1(C) = p \quad \text{and} \quad
		\dim E_{-1}(C') = \dim E_{-1}(C) = n-p.
	\]
	The unit vector $e_1 = a / \norm{a}_C$ can be completed to a basis
	$e_1, \dotsc, e_n \in \HR$ that is orthonormal for the inner product
	$\inner{\argbl}{\argbl}_C$, in such a way that $e_1, \dotsc, e_p \in E_1(C)$ and
	$e_{p+1}, \dotsc, e_n \in E_{-1}(C)$. Choose a second basis $e_1', \dotsc, e_n'
	\in \HR$ with $e_1' = a / \norm{a}_{C'}$ that is similarly adapted to
	$\inner{\argbl}{\argbl}_{C'}$ and $C'$, and let $g \in \Aut(\HR)$ be the unique
	automorphism such that $g e_i = e_i'$ for $i = 1, \dotsc, n$. Then obviously $ga =
	a$. As in the proof of \cref{lem:Weil}, one shows that $C' = g C g^{-1}$ and
	$g \in G(\RR)$, which then implies that $g \in G_a(\RR)$ because $ga = a$.
\end{proof}

\newpar
The orthogonal complement of $a$ relative to $Q$ is the subspace
\[
	\HQ' = \menge{v \in \HQ}{Q(a,v) = 0} = \menge{v \in \HQ}{\inner{a}{v}_C = 0}.
\]
Evidently, $\HQ = \QQ a \oplus \HQ'$. It is also easy to see that $C(\HR') \subseteq
\HR'$; consequently, the restriction of $C$ to $\HQ'$ is a Weil operator for the pair
$(\HQ', Q)$. We denote by 
\[
	K_a = \menge{g \in G_a(\RR)}{gC = Cg}
\]
the resulting maximal compact subgroup; note that $K_a = K \cap G_a(\RR)$.

\begin{proposition} \label{prop:Siegel-Ga}
	Any Siegel set in $G_a(\RR)$ for the maximal compact subgroup $K_a$ is contained
	in finitely many $G(\QQ)$-translates of a Siegel set in $G(\RR)$ $($for $K)$.
\end{proposition}

\begin{proof}
	The criterion in \cref{thm:Siegel-reduction} reduces the problem to the
	following concrete statement: suppose that $g \in G_a(\RR)$ is an element with the
	property that
	\[
		(v,w) \mapsto \inner{v}{w}_{gCg^{-1}} = Q(v, gCg^{-1} w)
	\]
	is $t$-reduced relative to an ordered basis $v_1, \dotsc, v_{n-1} \in \HQ'$; then it
	is possible to add the vector $a$ to the basis (in one of the $n$ possible places)
	and still keep the inner product $t$-reduced. This is completely elementary. To
	simplify the notation, let us agree to write $\inner{v}{w}$ and
	$\norm{v}$ instead of $\inner{v}{w}_{gCg^{-1}}$ and $\norm{v}_{gCg^{-1}}$. Without
	loss of generality, we may assume that $t \geq 1$. Recall that $\norm{a}^2 =
	Q(a,a)$. Since 
	\[
		\inner{v_i}{a} = Q(v_i, g C g^{-1} a) = Q(v_i, a) = 0,
	\]
	the second and third condition in the definition are trivially satisfied. For the
	first condition, note that for every $i = 1, \dotsc, n-1$, at least one of the
	inequalities
	\[
		\norm{a}^2 \leq t \norm{v_i}^2 \quad \text{or} \quad
		\norm{v_i}^2 \leq t \norm{a}^2
	\]
	will be true (because $t \geq 1$). Consequently, there is some value of $i \in
	\{1, \dotsc, n-1\}$ with the property that
	\[
		\norm{v_i}^2 \leq t \norm{a}^2 \quad \text{and} \quad
		\norm{a}^2 \leq t \norm{v_{i+1}}^2.
	\]
	But this is saying exactly that our inner product is $t$-reduced relative to the
	ordered basis
	\[v_1, \ldots, v_i, a, v_{i+1}, \ldots, v_{n-1} \in \HQ.\]
\end{proof}

\section{Definable structures on flat bundles}

\newpar

Let $X$ be a nonsingular complex algebraic variety, and let $\shE$ be a locally free
$\OX$-module of finite rank with a flat holomorphic connection $\nabla \colon \shE
\to \OmX^1 \tensor_{\OX} \shE$. We denote by $p \colon E \to X$ the associated
holomorphic vector bundle. Recall that $E$ is actually an algebraic vector bundle;
the algebraic structure on $E$ is uniquely determined by the flat connection. Let us
briefly review the construction. Choose an embedding $X \into Y$ into a complete
nonsingular variety, such that $D = Y \setminus X$ is a simple normal crossing
divisor. Let $(\shEt, \nabla)$ be Deligne's canonical extension of the pair $(\shE,
\nabla)$; up to isomorphism, it is determined by the following two conditions:
\begin{enumerate}
\item $\shEt$ is a locally free $\OY$-module with a flat logarithmic connection
	\[
		\nabla \colon \shEt \to \OmY^1(\log D) \tensor_{\OY} \shEt,
	\]
	such that $(\shEt, \nabla) \restr{X} \cong (\shE, \nabla)$. 
\item For each irreducible component $D_j$ of the divisor $D$, the pointwise
	eigenvalues of the residue operator
	\[
		\Res_{D_j} \nabla \in \End \bigl( \shEt \restr{D_j} \bigr)
	\]
	are complex numbers whose real part is contained in the interval $[0,1)$.
\end{enumerate}

The canonical extension has the following simple description in local coordinates.
Let $U \cong \Delta^n$ be an open neighborhood of a point $y \in Y$, with local
holomorphic coordinates $t_1, \dotsc, t_n$ centered at $y$, such that the divisor $D
\cap U$ is defined by the equation $t_1 \dotsb t_k = 0$. Let $V$ be the fiber of the
vector bundle $\shEt$ at the point $y$. Then there is a unique holomorphic
trivialization
\[
	\shEt \restr{U} \cong \shO_U \tensor_{\CC} V
\]
that restricts to the identity on $V$ at the point $y$, such that the logarithmic
connection takes the form
\[
	\nabla(1 \tensor v) = \sum_{j=1}^k \frac{\dt_j}{t_j} \tensor R_j v
\]
for commuting operators $R_1, \dotsc, R_k \in \End(V)$, all of whose eigenvalues have
real part in $[0,1)$; here $R_j$ is the residue operator along $t_j = 0$. It is easy
to see that
\[
	e^{-2 \pi i \sum_{j=1}^k \log t_j R_j} (1 \tensor v)
\]
defines a multivalued flat section of $(\shE, \nabla)$ on $U \cap X$. The monodromy
transformation around the divisor $t_j = 0$ is therefore described by the operator
\[
	e^{-2 \pi i R_j} \in \GL(V).
\]
In particular, the following two conditions are equivalent:
\begin{aenumerate}
\item The eigenvalues of the local monodromy transformations around the components of
	$D$ are complex numbers of absolute value $1$.
\item For each irreducible component $D_j$ of the divisor $D$, the pointwise
	eigenvalues of the residue operators $\Res_{D_j} \nabla$ are real numbers.
\end{aenumerate}

Since $Y$ is complete, the holomorphic vector bundle $p \colon \Et \to Y$ associated
to the locally free sheaf $\shEt$ has a unique algebraic structure; the algebraic
structure on the bundle $E$ is obtained by restriction. As before, we give $Y$ and $\Et$
the structure of $\RRanexp$-definable complex manifolds extending their algebraic
structures; this induces definable complex manifold structures on $X$ and $E$.  The former is the canonical algebraic definable structure $X^\mathrm{def}$ on $X$ and we call the latter the \emph{algebraic} definable structure on $E$.  It is uniquely described as that for which any holomorphic section $s$ of $p:\tilde E\to Y$ over an open subset $U'\subset Y$ restricts to a definable map on any definable $U\subset U'\cap X$ which has compact closure in $U'$.

\newpar
The holomorphic vector bundle $E$ naturally comes equipped with another $\RRanexp$-definable structure coming from the flat coordinates which we construct as follows.  The subsheaf $\shE^\nabla\subset \shE$ of flat sections is a complex local system.  By definable triangulation, the definable topological space $X^\mathrm{def}$ admits a definable atlas $\{(U_i,\phi_i)\}$ with each $U_i$ simply connected.  Choosing a basis of flat sections $s_1,\ldots,s_r$ of $\shE^\nabla|_{U_i}$, we therefore obtain a holomorphic trivialization $\psi_i:U_i\times\C^r \xrightarrow{\cong} p^{-1}(U_i)$ via $(u,z_1,\ldots,z_r)\mapsto \sum_i z_is_i(u)$.  Moreover, on intersections $U_{ij}:=U_i\cap U_j$ the transition functions are constant:  there are $g_{ij}\in \GL_r(\C)$ such that $\psi_i=\psi_j\circ(\id\times g_{ij})$.  The \emph{flat} definable structure on $E$ is then given by the definable complex manifold atlas $\{(p^{-1}(U_i),(\phi_i\times\id)\circ\psi_i^{-1})\}$.  It is uniquely characterized by the property that any flat section $s$ of $p:E\to X$ over a definable open subset $U\subset X$ is definable.

\newpar
From now on, we assume that the local monodromy transformations around the components
of $D$ have eigenvalues of absolute value $1$; in this case we say that $E$ has norm one eigenvalues at infinity. Under this assumption, the two definable structures from the previous two paragraphs are equivalent.
\begin{proposition}[\emph{cf.}~{\cite[Theorem 1.2]{flatbun}}]Let $X$ be a nonsingular complex algebraic variety and $E$ a holomorphic flat vector bundle over $X$ with norm one eigenvalues at infinity in the above sense.  Then the flat and algebraic definable complex manifold structures on $E$ are equivalent.
\end{proposition}

The idea of the proof can be seen from the construction of the Deligne canonical extension as above.  An algebraic frame for $\shE$ extends to an algebraic frame for the canonical extension $\tilde\shE$ on the compactification $Y$.  Thus, locally on the boundary it is related by a matrix of restricted analytic functions to a basis of sections of the form $e^{-2 \pi i \sum_{j=1}^k \log t_j R_j} (1 \tensor v)$.  This basis is in turn related to the flat basis by the matrix $e^{-2 \pi i \sum_{j=1}^k \log t_j R_j}$, which is $\RRanexp$-definable on bounded angular sectors.  Indeed, the functions $\log t_j$ are $\RRanexp$-definable on bounded angular sectors in polydisk neighborhoods of the boundary in $X^\mathrm{def}$.  If $R_j=R_j^{ss}+R_j^u$ is the Jordan decomposition, then  $e^{-2 \pi i \sum_{j=1}^k \log t_j R_j^u}$ is polynomial in the $\log t_j$, while $e^{-2 \pi i \sum_{j=1}^k \log t_j R_j^{ss}}$ is $\RRanexp$-definable since the $R_j^{ss}$ are real.

In particular, we have the following concrete corollary, which we will use.
\begin{corollary}\label{prop:flat-definable}
	Let $Z$ be a complex manifold, and let $f \colon Z \to X$ be a holomorphic mapping
	that is moreover $\RRanexp$-definable. Let $\sigma \in \Gamma(Z, \fu \shE)$ be a
	holomorphic section with $\nabla \sigma = 0$. Then the resulting function $\sigma
	\colon Z \to E$ is $\RRanexp$-definable with respect to the algebraic definable structure.
\end{corollary}

\section{Proof of the main theorem}

\newpar
We now come to the proof of \cref{thm:main}. Let $X$ be a nonsingular complex
algebraic variety, $\shH$ a polarized integral variation of Hodge structure on $X$,
of even weight $2k$. Fix a base point $x_0 \in X$ and let $\HZ = \shH_{\ZZ, x_0}$;
this is a free $\ZZ$-module of finite rank, which comes with a symmetric bilinear pairing
\[
	Q = Q_{x_0} \colon \HZ \tensor_{\ZZ} \HZ \to \ZZ.
\]
As usual, we set $\HQ = \HZ \tensor_{\ZZ} \QQ$ and $\HR = \HZ \tensor_{\ZZ} \RR$; for
simplicity, we shall use the notation $C = C_{x_0} \in \End(\HR)$ for the Weil operator
of the Hodge structure at the point $x_0$. In particular,
\[
	(v,w) \mapsto \inner{v}{w}_C = Q(v, Cw)
\]
is a positive definite inner product on the vector space $\HR$.

\newpar
Let $D$ be the period domain parametrizing integral Hodge structures of weight $2k$
on $\HZ$ that are polarized by $Q$. Since the statement of \cref{thm:main} only
involves the Weil operator (instead of the full Hodge structure), it makes sense to
consider not the period domain $D$, but rather the associated symmetric space (which
is a quotient of $D$). Consider the $\QQ$-algebraic group
\[
	G = O(\HQ, Q),
\]
whose set of real points is the real Lie group
\[
	G(\RR) = \menge{g \in \GL(\HR)}{\text{$Q(gv,gw) = Q(v,w)$ for all $v,w \in \HR$}}.
\]
Recall that $G(\RR)$ acts transitively on the period domain $D$, and that $D =
G(\RR)/V$, where $V \subseteq G(\RR)$ is the stabilizer of the Hodge structure at
$x_0$. Clearly, the Weil operator satisfies $C \in G(\RR)$ and $C^2 = \id$. It is
easy to see that
\[
	K = \menge{g \in G(\RR)}{gC = Cg}
\]
is a maximal compact subgroup of $G(\RR)$ containing the compact subgroup $V$; by
\cref{lem:Weil}, the points of the quotient $G(\RR)/K$ can be identified with
Weil operators for the pair $(\HQ, Q)$, with the coset $g K$ corresponding to the
Weil operator $g C g^{-1}$. 

\newpar
Consider now the arithmetic subgroup
\begin{equation} \label{eq:Gamma}
	\Gamma = O(\HZ, Q) = G(\QQ) \cap \GL(\HZ).
\end{equation}
Note that $\Gamma$ is quite a bit larger than the monodromy group of the variation of
Hodge structure; this will be important in what follows. Instead of the usual period
mapping to $\Gamma \backslash D$, we consider the (Weil operator) period mapping
\begin{equation} \label{eq:Phi}
	\Phi \colon X \to \Gamma \backslash G(\RR)/K
\end{equation}
to the arithmetic quotient of the symmetric space $G(\RR)/K$. It associates to every
point $x \in X$ the Weil operator $C_x \in \Aut(E_x)$, viewed as an automorphism of
the fixed vector space $\HC$ by parallel transport; this is well-defined in the
quotient $\Gamma \backslash G(\RR)/K$ since $\Gamma$ contains the monodromy group of
the variation of Hodge structure. According to \cite[Theorem~1.3]{Bakker+Klingler+Tsimerman:TameTopology}, the mapping
$\Phi$ in \eqref{eq:Phi} is $\RRanexp$-definable. The main result
is stated for the usual period mapping into $\Gamma \backslash D$, but what is
actually proved in \cite[Theorem~4.1]{Bakker+Klingler+Tsimerman:TameTopology} is the $\RRanexp$-definability of \eqref{eq:Phi}.

\newpar
Let $\pi \colon \Xt \to X$ be the universal covering space
of $X$. Since the period mapping is locally liftable, there is a real-analytic mapping
$\Phit \colon \Xt \to G(\RR)/K$, unique up to a choice of base point, making the
following diagram commute:
\[
	\begin{tikzcd}
		\Xt \rar{\Phit} \dar{\pi} & G(\RR)/K \dar \\
		X \rar{\Phi} & \Gamma \backslash G(\RR)/K
	\end{tikzcd}
\]
We now extend the definability result to the vector bundle $p \colon E \to X$. On
$G(\RR)/K$, consider the trivial complex vector bundle $G(\RR)/K \times \HC$, where
$\HC = \HZ \tensor_{\ZZ} \CC$. The arithmetic group $\Gamma$ acts on this bundle via
the formula
\[
	\gamma \cdot (gK, v) = \bigl( \gamma gK, \gamma(v) \bigr),
\]
and the quotient gives us a ``universal'' complex vector bundle
\[
	\Gamma \backslash \bigl( G(\RR)/K \times \HC \bigr) 
	\to \Gamma \backslash G(\RR)/K
\]
with fiber $\HC$. The pullback $\piu \shE$ has a canonical trivialization by
$\nabla$-flat sections, hence $\piu E \cong \Xt \times \HC$. The trivial morphism of
vector bundles
\[
	\Phit \times \id \colon \Xt \times \HC \to G(\RR)/K \times \HC
\]
therefore descends to a morphism of complex vector bundles
\[
	\begin{tikzcd}
		E \rar{\Phi_E} \dar{p} & \Gamma \backslash \bigl( G(\RR)/K \times \HC \bigr) \dar \\
		X \rar{\Phi} & \Gamma \backslash G(\RR)/K.
	\end{tikzcd}
\]

\begin{proposition} \label{prop:E-definable}
	The morphism of complex vector bundles 
	\[
		\Phi_E \colon E \to \Gamma \backslash \bigl( G(\RR)/K \times \HC \bigr)
	\]
	is $\RRanexp$-definable.
\end{proposition}

\begin{proof}
	Let $Y$ be a complete nonsingular variety containing $X$, such that $X = Y
	\setminus D$ for a simple normal crossing divisor $D$.  By the same argument as in
	\cite[\textsection\,4.1]{Bakker+Klingler+Tsimerman:TameTopology}, the problem is local on $Y$, and so we may assume that $Y =
	\Delta^n$, with holomorphic coordinates $t_1, \dotsc, t_n$, and that the divisor
	$D$ is defined by the equation $t_1 \dotsm t_k = 0$. Let $\HH = \menge{z \in
	\CC}{\Im z > 0}$ and $\Sigma = \menge{z \in \HH}{0 \leq \Re z \leq 1}$. 
	By \cite[Theorem~1.5]{Bakker+Klingler+Tsimerman:TameTopology}, there is a subalgebraic Siegel set $\Sieg
	\subseteq G(\RR)$ for the maximal compact subgroup $K$, and a finite set $A
	\subseteq G(\QQ)$, such that the image of
	\[
		\Phit \colon \Sigma^k \times \Delta^{n-k} \to G(\RR)/K
	\]
	is contained in $A \cdot \Sieg$. We get the following commutative diagram:
	\[
		\begin{tikzcd}
			\Sigma^k \times \Delta^{n-k} \rar{\Phit} \dar[hook] \arrow[bend
			right=40,shift right=5ex,swap]{dd}{\pi}
				& A \cdot \Sieg \dar[hook] \\
			\HH^k \times \Delta^{n-k} \rar{\Phit} \dar{\pi} & G(\RR)/K \dar \\
			(\dst)^k \times \Delta^{n-k} \rar{\Phi} & \Gamma \backslash G(\RR)/K
		\end{tikzcd}.
	\]
	Now $\pi \colon \Sigma^k \times \Delta^{n-k} \to (\dst)^k
	\times \Delta^{n-k}$ is $\RRanexp$-definable, and so
	\cref{prop:flat-definable} implies that the isomorphism of complex
	vector bundles
	\[
		\piu E \cong \Sigma^k \times \Delta^{n-k} \times \HC
	\]
	is actually $\RRanexp$-definable. Since the morphism of trivial bundles
	\[
		\Phit \times \id \colon 
		\Sigma^k \times \Delta^{n-k} \times \HC \to (A \cdot \Sieg) \times \HC
	\]
	is obviously $\RRanexp$-definable, we get the desired result.
\end{proof}

\newpar
We are ready to prove a first definability result for self-dual vectors in a single
$\Gamma$-orbit. Suppose that we have an integral vector $a \in \HZ$ such
that $Ca = a$. We are interested in self-dual classes in the orbit $\Gamma a
\subseteq \HZ$. As noted after \cref{lem:Weil}, points of the symmetric space $G(\RR)/K$
correspond to Weil operators for $(\HQ, Q)$; we identity a coset $gK$ with the Weil
operator $C_g = g C g^{-1}$.

\begin{proposition} \label{prop:definable-orbit}
	Let $a \in \HZ$ be a nonzero integral vector with $Ca = a$. The set
	\[
		\menge{\Gamma (gK, v) \in \Gamma \backslash \bigl( G(\RR)/K \times \HC
		\bigr)}{\text{$v \in \Gamma a$ and $C_g v = v$}}
	\]
	is $\RRalg$-definable.
\end{proposition}

\begin{proof}
	We introduce the additional subgroups
	\[
		K_a = G_a(\RR) \cap K \quad \text{and} \quad
		\Gamma_a = G_a(\QQ) \cap \Gamma.
	\]
	For the same reason as before, $K_a$ is a maximal compact subgroup of $G_a(\RR)$,
	and $\Gamma_a$ is an arithmetic subgroup of $G_a(\QQ)$. \cref{lem:orbit-Ga}
	shows that the image of 
	\[
		G_a(\RR)/K_a \into G(\RR)/K
	\]
	consists of all cosets $gK$ whose corresponding Weil operator $C_g =
	gCg^{-1}$ satisfies $C_g a = a$. According to \cref{prop:Siegel-Ga} and
	\cite[Theorem~1.2]{Bakker+Klingler+Tsimerman:TameTopology}, the morphism of arithmetic quotients
	\[
		\Gamma_a \backslash G_a(\RR)/K_a \to \Gamma \backslash G(\RR)/K
	\]
	is $\RRalg$-definable. This morphism has a well-defined lifting
	\[
		i \colon \Gamma_a \backslash G_a(\RR)/K_a \to 
			\Gamma \backslash \bigl( G(\RR)/K \times \HC \bigr), \quad
			\Gamma_a h K_a \mapsto \Gamma (hK, a),
	\]
	which is $\RRalg$-definable for the same reason. In more detail, let $\Sieg$ be an
	arbitrary Siegel set in $G_a(\RR)$ with respect to the maximal compact subgroup
	$K_a$. The formula
	\[
		\tilde{i} \colon \Sieg \to G(\RR)/K \times \HC, \quad
		\tilde{i}(h) = (hK, a)
	\]
	gives us an $\RRalg$-definable local lifting of $i$. Because of
	\cref{prop:Siegel-Ga}, the composition of $\tilde{i}$ with the
	projection to $\Gamma \backslash(G(\RR)/K \times \HC)$ is definable;
	it follows that the mapping $i$ is itself $\RRalg$-definable.
	
	Now the image of $i$ is exactly the set we are interested in. Indeed, suppose that
	$\Gamma (gK, v) = \Gamma(hK, a)$ for some $g \in G(\RR)$, $h \in G_a(\RR)$, and $v \in
	\HC$. Then there is an element $\gamma \in \Gamma$ such that $g = \gamma h$ and $v =
	\gamma a$, and one easily deduces that $v \in \Gamma a$ and $C_g v = v$.
	In fact, the mapping $i$ is
	an embedding: if $\Gamma(hK, a) = \Gamma(h'K, a)$ for two elements $h,h' \in
	G_a(\RR)$, then there is some $\gamma \in \Gamma$ and some $k \in K$ such that $h'
	= \gamma h k$ and $\gamma a = a$; but then $\gamma \in \Gamma_a$, and therefore $k \in
	K_a$, and so the double cosets $\Gamma_a h' K_a = \Gamma_a h K_a$ are equal. The
	locus of self-dual classes in our given $\Gamma$-orbit is therefore an
	$\RRalg$-definable subset that is isomorphic to the smaller arithmetic quotient
	$\Gamma_a \backslash G_a(\RR)/K_a$. 
\end{proof}

\newpar
We now extend the above result to integral vectors $v \in \HZ$ with a fixed
self-intersection number $Q(v,v)$. 

\begin{proposition} \label{prop:definable-all}
	Let $q \in \NN$ be a positive integer. Then the set
	\[
		\menge{\Gamma (gK, v) \in \Gamma \backslash \bigl( G(\RR)/K \times \HC
		\bigr)}{\text{$v \in \HZ$, $Q(v,v) = q$, and $C_g v = v$}}
	\]
	is $\RRalg$-definable.
\end{proposition}

\begin{proof}
	The crucial point is that $\Gamma$ acts on the set $\menge{v \in
	\HZ}{Q(v,v) = q}$ with only \emph{finitely} many orbits; this makes the result
	a direct consequence of \cref{prop:definable-orbit}. The finiteness of
	the number of $\Gamma$-orbits follows from \cite[Satz~30.2]{Kneser}; since Kneser
	works in much greater
	generality, let us briefly explain how to deduce the statement we need.
	The quadratic form $v \mapsto Q(v,v)$ makes $\HZ$ into a lattice
	in the $\QQ$-vector space $\HQ$, and an integral vector $v \in \HZ$ with $Q(v,v) =
	q$ defines an isometry $v \colon [q] \to \HZ$, where $[q]$ means the
	lattice $\ZZ$ with the quadratic form $n \mapsto q n^2$. (Kneser calls this a
	``Darstellung'' of $[q]$ in $\HZ$.) Now $[\ell]$ is nondegenerate because $\ell
	\geq 1$, and \cite[Satz~30.2]{Kneser} guarantees that there are only finitely many
	equivalence
	classes of such isometries. But since $\Gamma = O(\HZ, Q)$, two vectors $v,v' \in
	\HZ$ are in the same equivalence class, in the sense of \cite[Definition~30.1]{Kneser},
	exactly when there is an element $\gamma \in \Gamma$ such that $v' = \gamma v$.
\end{proof}

\newpar
Finally, we assemble all the pieces and prove \cref{thm:main}.

\begin{proof}[Proof of Theorem~\ref*{thm:main}]
	The polarized integral variation of Hodge structure $\shH$ on $X$ gives rise to a
	kind of period mapping
	\[
		\Phi \colon X \to \Gamma \backslash G(\RR)/K
	\]
	that, up to the action by $\Gamma$, associates to every point $x \in X$ the Weil
	operator $C_x$ of the corresponding Hodge structure. We already know that $\Phi$
	is $\RRanexp$-definable \cite[Theorem~1.2]{Bakker+Klingler+Tsimerman:TameTopology}. We also have a morphism of complex
	vector bundles
	\[
		\Phi_E \colon E \to \Gamma \backslash \bigl( G(\RR)/K \times \HC \bigr)
	\]
	from the algebraic vector bundle $p \colon E \to X$ to the ``universal'' vector
	bundle on the right. We also know that $\Phi_E$ is $\RRanexp$-definable (by
	\cref{prop:E-definable}). The two morphisms fit into the following
	commutative diagram:
	\[
		\begin{tikzcd}
			E \rar{\Phi_E} \dar{p} & \Gamma \backslash \bigl( G(\RR)/K \times \HC \bigr) \dar \\
			X \rar{\Phi} & \Gamma \backslash G(\RR)/K.
		\end{tikzcd}
	\]
	Fix a positive integer $q \in \NN$.  By \cref{prop:definable-all},
	the set
	\[
		\menge{\Gamma (gK, v) \in \Gamma \backslash \bigl( G(\RR)/K \times \HC
		\bigr)}{\text{$v \in \HZ$, $Q(v,v) = q$, and $C_g v = v$}}
	\]
	is $\RRalg$-definable, and so its preimage under $\Phi_E$ is an
	$\RRanexp$-definable subset of $E$. Since it is easy to see that a point
	$(x,v) \in E$ lies in the preimage exactly when $v \in E_x$ is integral and
	satisfies $Q_x(v,v) = q$ and $C_x v = v$, we get the result.
\end{proof}

\section{Additional results}

\newpar
In this section, we prove the two variants of the main theorem stated in the
introduction. The idea is simple enough: we tensor a given integral variation of Hodge
structure by an auxiliary Hodge structure of weight $1$ or $2$, and then apply
\cref{thm:main}. The first Hodge structure that we need is the following.

\begin{example} \label{ex:HS1}
	Consider the Hodge structure on the first cohomology of the elliptic
	curve $\CC/(\ZZ \oplus \ZZ i)$. Concretely, this is an integral Hodge structure of
	weight $1$ on the free $\ZZ$-module $\ZZ^{\oplus 2}$, whose Hodge decomposition is
	given by
	\[
		\CC^{\oplus 2} = \CC(1,i) \oplus \CC(1,-i).
	\]
	The Hodge structure is polarized by the skew-symmetric bilinear form 
	\[
		\ZZ^{\oplus 2} \tensor \ZZ^{\oplus 2} \to \ZZ, \quad
			\bigl( (a_1,a_2), (b_1,b_2) \bigr) \mapsto a_1 b_2 - a_2 b_1, 
	\]
	and the Weil operator is easily seen to be the operator $(a_1, a_2) \mapsto
	(a_2,-a_1)$; note that it happens to preserve the integral structure in this case.
\end{example}

Now suppose that $H$ is a polarized integral Hodge structure of odd weight $2k-1$.
Let $Q \colon \HZ \tensor_{\ZZ} \HZ \to \ZZ$ be the skew-symmetric bilinear form
giving the polarization, and let $C \in \End(\HR)$ be the Weil operator (which now
satisfies $C^2 = -\id$). After taking the tensor product with the above Hodge
structure of weight $1$, we obtain a polarized integral Hodge structure $\tilde{H}$
of weight $2k$ on
\[
	\tilde{H}_{\ZZ} = \HZ \oplus \HZ,
\]
polarized by the symmetric bilinear form $\tilde{Q} \bigl( (a_1,a_2), (b_1,b_2) \bigr) =
Q(a_1,b_2) - Q(a_2,b_1)$, and with Weil operator $\tilde{C}(a_1,a_2) = (Ca_2,
-Ca_1)$. Evidently,
\[
	\tilde{H}_{2q}^+ = \menge{(a_1,a_2) \in \HZ \oplus \HZ}{%
	\text{$a_1 = Ca_2$ and $Q(a_1,a_2) = q$}};
\]
for polarized integral variations of Hodge structure of \emph{odd} weight,
\cref{cor:pairs} is therefore an immediate consequence of \cref{thm:main}.

\newpar
The remaining assertions concern variations of Hodge structure of even weight.
We can deal with them by the same method, using the following Hodge structure.

\begin{example} \label{ex:HS2}
	Consider now the symmetric square of the Hodge structure in \cref{ex:HS1}.
	Concretely, we get an integral Hodge structure of weight $2$ on the free
	$\ZZ$-module $\ZZ^{\oplus 3}$, whose Hodge decomposition is
	\[
		\CC^{\oplus 3} = \CC(1, 2i, -1) \oplus \CC(1,0,1) \oplus \CC(1, -2i, -1).
	\]
	The Hodge structure is polarized by the symmetric bilinear form 
	\[
		\ZZ^{\oplus 3} \tensor \ZZ^{\oplus 3} \to \ZZ, \quad
			\bigl( (a_1,a_2,a_3), (b_1,b_2,b_3) \bigr) \mapsto a_1 b_3 + a_3 b_1 - a_2
			b_2,
	\]
	and the Weil operator is easily seen to be the operator $(a_1,a_2,a_3) \mapsto
	(a_3,-a_2,a_1)$, which again preserves the integral structure.
\end{example}

Suppose that $H$ is a polarized integral Hodge structure of even weight $2k$,
with polarization $Q \colon \HZ \tensor_{\ZZ} \HZ \to \ZZ$ and Weil operator $C \in
\End(\HR)$. After taking the tensor product with the Hodge structure in
\cref{ex:HS2}, we obtain a polarized integral Hodge structure $\tilde{H}$ of
weight $2k+2$ on
\[
	\tilde{H}_{\ZZ} = \HZ \oplus \HZ \oplus \HZ,
\]
polarized by the symmetric bilinear form 
\[
	\tilde{Q} \bigl( (a_1,a_2,a_3), (b_1,b_2,b_3) \bigr) 
	= Q(a_1,b_3) + Q(a_3,b_1) - Q(a_2,b_2), 
\]
and with Weil operator $\tilde{C}(a_1,a_2,a_3) = (Ca_3, -Ca_2, Ca_1)$. This time around,
\[
	\tilde{H}_{q}^+ = \menge{(a_1,a_2,a_3) \in \HZ \oplus \HZ \oplus \HZ}{\text{$a_1 = Ca_3$, $Ca_2 = -a_2$, and}\
\text{$2 Q(a_1,a_3) - Q(a_2,a_2) = q$}}.
\]
We now obtain \cref{cor:pairs} for polarized integral variations of Hodge
structure of \emph{even} weight by looking at triples of the form $(a_1,0,a_3)$, and
\cref{cor:anti-self-dual} by looking at triples of the form $(0,a_2,0)$.

\section{Motivation from string theory}

\newpar
A motivation for studying the locus of self-dual integral Hodge classes stems from string theory. 
String theory is a candidate theory of quantum gravity that unifies Einstein's theory of general relativity 
and quantum field theory. Quantum consistency forces the string to travel through a higher-dimensional 
space-time manifold, extending beyond the four space-time dimensions that we currently observe in our universe.  
In prominent variants of string theory this implies that either six or eight extra dimensions need to be 
present. These extra dimensions are often considered to be on a tiny compact manifold. Particularly well-studied 
choices are Calabi-Yau manifolds, which are defined to be K\"ahler manifolds that admit a Ricci-flat metric. 
While it is not known which Calabi-Yau manifold one should pick, it has been studied intensively how the 
physical four-dimensional theory can be determined after making a choice. 

\newpar
We now describe one example from physics that originally suggested the result in
\cref{thm:main}. 
Let $Y$ be a compact polarized Calabi-Yau manifold of complex dimension $D$, with $D=3,4$ being the cases most 
relevant in the string theory application. One can associate a family of manifolds $Y_t$ to $Y$ that is obtained by deforming its complex structure. It is shown by the 
 Bogomolov-Tian-Todorov theorem \cite{GTian,Todorov} that the Kuranishi space of $Y$ is unobstructed. Hence $Y_t$ varies over a finite-dimensional moduli space $\mathcal{M}$ if one demands that all $Y_t$ 
are Calabi-Yau manifolds. For polarized Calabi-Yau manifolds of complex 
dimension $D$ the moduli space $\mathcal{M}$ is quasi-projective \cite{Viehweg} and of complex dimension $h^{D-1,1} = \dim H^{D-1,1}(Y)$. The 
existence of such a moduli space leads to several physical problems when using such $Y_t$ as backgrounds of 
string theory. In particular, one finds modifications of Newton's law or Einstein's equations that are in  
contradiction with observations. To avoid this immediate conclusion further ingredients known as background fluxes can be introduced.
These fluxes are integral classes in the cohomology of $Y$. Compared with the general considerations above, we 
thus identify $X= \mathcal{ \widehat  M}$, where $ \mathcal{ \widehat  M}$ is the resolution of $\mathcal{M}$ \cite{Hironaka}.
Furthermore, we set $H_\mathbb{Z}=H^{D}(Y,\mathbb{Z})$ and $Q = \int_{Y} v \wedge w$. Introducing a Weil operator $C$
 acting on $H^{p,q}(Y_t)$ with $i^{p-q}$, we define a norm $\|w \|^2 = Q(\bar w, C w)$. 
 
 \newpar 
The best understood string theory settings with integral fluxes are obtained from Type IIB string theory \cite{GVW,GKP}. Let 
us consider this ten-dimensional theory on a Calabi-Yau manifold $Y$ of complex dimension three. In this string theory setting one is also free to chose in addition to $Y$ two 
integral three-forms $F,H \in H^{3}(Y,\mathbb{Z})$, which set the flux background. 
They naturally combine to a complex three-form $ \mathcal{G}= F - \tau H$ with $\tau \in \mathbb{C}$. The fluxes $F,H$ 
are constrained by a consistency condition 
\begin{equation} \label{tadpole}
 Q(F,H) =  \ell \ , 
 \end{equation}
 where $\ell$ is a fixed positive rational number that can 
be derived for a given setting. This condition is known as a tadpole cancellation condition and plays a crucial 
role in finding consistent solutions of string theory. 
Furthermore, the presence of $\mathcal{G}$ impacts 
the physical four-dimensional theory by giving rise to an energy potential $V( \mathcal{G})$, which generally changes for different choices $Y_t$
within the family. Concretely, it takes the form \cite{GKP}
\begin{equation}
    V (\mathcal{G}) =\kappa \|\mathcal{G}^- \|^2 \ , 
\end{equation}
where $C \mathcal{G}^- = - i \mathcal{G}^-$ and $\kappa$ is $\tau$-dependent but constant over $\mathcal{M}$. The loci in $\mathcal{M}$ that minimize this energy potential with $\mathcal{G}^- = 0$ have been shown to be consistent background solutions of Type IIB string theory \cite{GKP}. It has been a long-standing question of whether or not the number of distinct $H,F$ with $\mathcal{G}^-=0$ and \eqref{tadpole} is finite.

\newpar
A more general setting that leads to a similar question arises in a geometric higher-dimensional version of Type IIB string theory known as F-theory \cite{Vafa:1996xn,Denef:2008wq}. 
In F-theory the extra dimensions are constrained to reside on an eight-dimensional compact manifold to extract a four-dimensional physical theory. 
The consistency equations for such twelve-dimensional 
string backgrounds admit solutions that are (conformal) Calabi-Yau manifolds $Y$ of complex dimension $D=4$ that admit a four-form flux background. 
Let us consider $v \in H^{4}(Y,\mathbb{Z})$ and assume that $v$ is primitive with respect to the K\"ahler form $J$ of $Y$. The condition \eqref{tadpole} generalizes to $Q(v,v)=\ell$, which is the consistency condition every solution 
to F-theory with a compact $Y$ has to satisfy. A non-trivial $v$ induces again an energy potential
\begin{equation}
 V(v) = \lambda \| v_- \|^2\ ,
\end{equation}
where $C v_- = -v_-$ and $\lambda$ is a constant. 
 $V(v)$ changes in the family $Y_t$ and hence is a function on the moduli space $\mathcal{M}$. The self-dual loci  in $\mathcal{M}$ are by definition those that satisfy 
 $C v = v$ and hence minimize  $V(v)$. They comprise consistent solutions to F-theory and are of central interest to some of the most prominent scenarios on realizing 
 our four-dimensional universe in string theory. Each choice of $v$ satisfying these consistency conditions can imply different values for physical observables. It is thus of profound  
 importance to know if there are infinitely many choices for $v$. 

\newpar
Finiteness statements about the set of self-dual $v \in H^{4}(Y,\mathbb{Z})$ with $Q(v,v)=\ell$ have been conjectured in \cite{Douglas:2003um,Acharya:2006zw}. In order to provide evidence for these statements and to estimate the number of distinct solutions it was suggested in \cite{Ashok:2003gk,Denef:2004ze} to  
introduce a critical point density on the moduli space. Mathematically rigorous proofs on estimating this 
density were given in \cite{Douglas:2004zu,Douglas:2004kc,Douglas:2005df}. Strong finiteness results have been shown in \cite{Douglas:2006zj,Lu:2009aw} for a certain index counting solutions to the self-duality relations for Calabi-Yau manifolds by applying a Gauss-Bonnet-Chern theorem on the moduli space. In this work we have given an affirmative answer to the finiteness conjectures
without using a density function. We have also shown that in the complex structure moduli space there is no need to introduce a refined notion of physically distinct vacua to ensure finiteness as suggested in \cite{Acharya:2006zw}. The finiteness statement is centrally based on the definability of the period mapping which suffices to exclude the pathological examples 
discussed in \cite{Acharya:2006zw}. In dimension $1$, it is possible to prove
\cref{thm:main} along the lines of \cite{Cattani+Deligne+Kaplan}, by using more details about the
$\SL(2)$-orbit theorem \cite{letter}, see also \cite{Grimm:2020cda} for a sketch of the argument. In higher dimensions, this kind of
argument looks completely infeasible.


\end{document}